\documentclass[british]{scrartcl}

\title {A stable splitting of factorisation homology of generalised surfaces}
\author{Florian Kranhold}
\date{}

\input{preamble.tex}

\usepackage{scrlayer-scrpage}
\lohead{F.\ Kranhold}
\rohead{A stable splitting of factorisation homology}
\begin{document}

\maketitle

\let\svthefootnote\thefootnote
\let\thefootnote\relax\footnotetext{\acr{MSC} classes: \textsc{55r80, 57s05, 57r15, 55p47, 55p48, 18n70}}
\let\thefootnote\svthefootnote

\begin{abstract}
  For a manifold $W$ and an $\dE d$-algebra $A$, the factorisation homology
  $\smash{\int_W A}$ can be seen as a generalisation of the classical
  configuration space of labelled particles in $W$. It carries an action by the
  diffeomorphism group $\Diff_\partial(W)$, and for the generalised surfaces
  $W_{g,1}\coloneqq (\#^g S^n\times S^n)\setminus\mathring D^{2n}$, we have
  stabilisation maps among the quotients
  $\smash{\int_{\raisebox{1px}{\scriptsize
        $\smash{\dW{\smash{g,1}}}$}} A\sslash\Diff_\partial(W_{g,1})}$ which
  increase the genus $g$.  In the case where a highly-connected tangential
  structure $\theta$ is taken into account, we describe its stable homology in
  terms of the iterated bar construction $\B^{2n}A$ and a tangential Thom
  spectrum $\MT\theta$. We also consider the question of homological
  stability.\looseness-1
\end{abstract}

\section{Motivation and overview}
\subsection{Stable factorisation homology of oriented surfaces}

Let $S$ be a smooth oriented surface, possibly with boundary. For each
$r\ge 0$, the topological group $\Diff^+_\partial(S)$ of orientation-preserving
diffeomorphisms fixing the boundary acts on the configuration space
$C_r(\mathring S)$ of $r$ unordered particles in the interior $\mathring S$, and
the homotopy quotient $C_r(\mathring S)\sslash \Diff^+_\partial(S)$ is a moduli space for
surfaces that are diffeomorphic to $S$ and carry $r$ permutable punctures in
their interior.\looseness-1

If $S_{g,1}$ is an oriented surface of genus $g\ge 0$ with one boundary curve,
then we obtain maps
$C_r(\mathring S_{g,1})\sslash \Diff^+_\partial(S_{g,1})\to C_r(\mathring S_{g+1,1})\sslash
\Diff^+_\partial(S_{g+1,1})$ by forming the boundary-connected sum with
$S_{1,1}$ that carries no particles. It follows essentially from Harer’s
stability theorem \cite{Harer-1984} that the sequence of these maps is
homologically stable. Moreover, the stable homology has been described by
\cite{CFB-Tillmann-2001} in terms of an ‘undecorated’ part
$\on{hocolim}_{g\to \infty} \B\Diff^+_\partial(S_{g,1})$ and
$\B(\frS_r\wr \SO(2))$.\looseness-1

A more recent work \cite{Bonatto} studies generalisations of these configuration
spaces, which capture more of the local structure of the surface: We fix a
framed $E_2$-algebra $A$, i.e.\ a space on which the operad
$E_2\rtimes \on{SO}(2)$ acts, and consider, for each oriented surface $S$ as
above, the space $\int_S A$ of oriented embeddings
$\coprod^r \mathring D^2\hookrightarrow S$, for arbitrary $r\ge 0$, where each
disc carries a label in $A$ and where the framed $E_2$-action on $A$ is balanced
with compositions of embeddings of discs as in \cite{Salvatore}. This is an
instance of factorisation homology developed in \cite[§\,5.5]{Lurie}. Again, we
have an action of the diffeomorphism group $\Diff_\partial^+(S)$ on $\int_S A$,
now by postcomposing embeddings of discs, and stabilisation maps
$\int_{S_{g,1}} A\sslash \Diff_\partial^+(S_{g,1}) \to \int_{S_{g+1,1}} A\sslash
\Diff_\partial^+(S_{g+1,1})$.\looseness-1

As before, it turns out that this sequence is homologically stable
\cite[Thm.\,\textsc{e}]{Bonatto} and that the stable homology splits as follows:
Bonatto constructs a geometrically-flavoured semi-simplicial space
$D_\bullet(A)$ out of $A$, whose geometric realisation is an $E_\infty$-algebra,
and identifies the group-completion of the $A_\infty$-algebra
$\smash{\coprod_{g\ge 0} \int_{\smash{S_{\smash{g,1}}}}A\sslash
  \Diff^+_\partial(S_{g,1})}$ with the infinite loop space
$\Omega^\infty\MT\SO(2)\times \Omega\B |D_\bullet(A)|$, see
\cite[Thm.\,\textsc{g}]{Bonatto}.  Here $\MT\SO(2)$ is the oriented tangential
Thom spectrum as in \cite{Madsen-Weiss}. Via the group-completion theorem from
\cite{McDuff-Segal}, see also \cite[Apx.\,\textsc{q}]{Friedlander-Mazur}, this
implies that the colimit of the above stabilisations splits homologically into
$\on{hocolim}_{g\to \infty}\B\Diff^+_\partial(S_{g,1})$ and
$\Omega\B|D_\bullet(A)|$.\looseness-1

\subsection{Aim and setting of this work}

The main point of this work is to establish a description of Bonatto’s second
factor in homotopy-theoretic terms.  Our methods to pursue this aim allow for a
more general input: Recall that for each dimension $d$ and each $d$-dimensional
tangential structure $\theta\colon L\to\on{BO}(d)$ (in the sense of
\cref{defi:tang}), there is an operad $\dE[\theta]d$ of $\theta$-framed
embeddings of $d$-dimensional discs, and for each $\theta$-framed manifold
$(W,\ell_W)$ and each $\dE[\theta]d$-algebra $A$, we can consider the
factorisation homology $\int^\theta_W A$, which geometrically is defined exactly
as in the case of surfaces. An explicit model as a two-sided bar construction is
given in \cite{Kupers-Miller}; we recall it in \cref{subsec:FH}.\looseness-1

As before, one can consider the \emph{moduli space} of such manifolds $W$ with
decorations in $A$, essentially by letting the $\theta$-framing vary and
quotienting out the action of the group $\Diff_\partial(W)$, see
\cref{constr:modDiff} for details. The result is called $W^\theta[A]$ and plays
a central rôle in this paper.  We are interested in these spaces for the
following reasons:
\begin{itemize}
\item They generalise at the same time the aforementioned generalised
  configuration spaces of surfaces from \cite{Bonatto} and the classifying
  spaces for punctured diffeomorphism groups for manifolds of arbitrary
  dimension from \cite{CFB-Tillmann-2001,Bonatto-2020}.
\item They appear in a homotopy fibre sequence
  $\int^\theta_W A\to W^\theta[A]\to \caM^\theta_\partial(W,\ell_W)$, where
  $\caM^\theta_\partial(W,\ell_W)$ is the classical moduli space of
  $\theta$-framed manifolds of type $(W,\ell_W)$ studied in
  \cite{Galatius-RW-2017,Galatius-RW-2018}; see \cref{prop:hofib} for details.
\item If $\Omega^{d}_\theta X$ is the $\theta$-framed loop space for some
  retractive space $X$ over $L$, see \cref{ex:loop} for a definition, then
  $W^\theta[\Omega^{d}_\theta X]$ is the moduli space
  $\caM^\theta_\partial(W,\ell_W)\ula{X}$ of $\theta$-framed manifolds of type
  $(W,\ell_W)$, together with a map to $X$ over $L$.
\end{itemize}
In the case of generalised surfaces
$W_{g,1}\coloneqq (\#^g S^n\times S^n)\setminus \mathring D^{2n}$, we obtain
stabilisation maps $\dW[\theta]{g,1}[A]\to \dW[\theta]{g+1,1}[A]$ by taking
boundary-connected sums with an ‘empty’ $W_{1,1}$.  Here we assume our
tangential structure to be spherical in the sense of \cite{Galatius-RW-2018},
i.e.\ we require that $S^{2n}$ admits a $\theta$-framing, in order to ensure the
existence of well-behaved $\theta$-framings on all $W_{g,1}$.  We study the
above stabilisation maps and their colimit – or, in other words, the
group-completion of the $A_\infty$-algebra
$\dW[\theta]{*,1}[A]\coloneqq \coprod_{g\ge 0} \dW[\theta]{g,1}[A]$.

\subsection{Results}

Let $\theta\colon L\to\B\on{O}(d)$ be a tangential structure with connected $L$. For any
$\dE[\theta]{d}$-algebra $A$, we denote by $\B^d UA$ the ite\-rated bar
construction of its underlying $E_d$-algebra, the latter depending on a choice
of basepoint $b_0\in L$. Our main result is the following:\looseness-1

\begingroup
\def\thetheo{\ref{thm:B}}
\begin{thm}
  Let $\theta\colon L\to \B\on{O}(2n)$ be a spherical tangential
  structure with $n$-connected $L$, and let $A$ be an
  $\dE[\theta]{2n}$-algebra. Then there is an $A_\infty$-action of $\Omega L$ on the
  spectrum $\Sigma^{\infty-2n}\B^{2n} UA$ and we have a weak equivalence of loop
  spaces\looseness-3
  \[\textstyle \Omega\B\dW[\theta]{*,1}[A] \simeq
    \Z\times \Omega^\infty_0\MT \theta\times \Omega^\infty
    \bigl(\bigl(\Sigma^{\infty-2n}\B^{2n} U A\bigr)_\h{\Omega L}\bigr).\]
\end{thm}
\endgroup

Here $\Omega^\infty_0\subseteq\Omega^\infty$ denotes the path-component of the
basepoint and $(-)_\h{\Omega L}$ denotes the homotopy quotient in the category
of spectra.  If $\theta$ is of the form $\B G\to \B\on{O}(2n)$ for some group
homomorphism $G\to \on{O}(2n)$, then we actually take homotopy orbits of a given
$G$-action on $\Sigma^{\infty-2n}\B^{2n} UA$. We point out that our assumptions
cover the aforementioned case of $2n=2$ and
$\theta\colon \B\SO(2)\to \B\on{O}(2)$.\looseness-1

Moreover, this result can be used to calculate the homology of
$\dW[\theta]{g,1}[A]$ in a range, as the stability results for moduli spaces of
surfaces \cite{Harer-1984,RW-2016} and high-dimensional manifolds
\cite{Galatius-RW-2018} rather directly imply the following
statement:

\begingroup
\def\thetheo{\ref{thm:A}}
\begin{thm}
  Let $\theta\colon L\to \B\on{O}(2n)$ be a spherical tangential structure with
  $\pi_1$-injective $\theta$, and let $A$ be an $\dE[\theta]{2n}$-algebra.  If
  $2n\ge 6$ (or $2n=2$ and $\theta$ is admissible to \cite[Thm.\,7.1]{RW-2016}),
  then the maps $W^\theta_{g,1}[A]\to W^\theta_{g+1,1}[A]$ induce isomorphisms
  in $H_i(-;\Z)$ for $i\le \frac12\kern1px g-\frac32$ (for $2n=2$, the slope is
  different, but coincides with the one from \cite[Thm.\,7.1]{RW-2016}).
\end{thm}
\endgroup

If $L$ is $n$-connected, as in \cref{thm:C}, then $\theta$ is in particular
$\pi_1$-injective. The case of $2n=2$ and $\theta$ being orientations is
covered by \cite[Thm.\,7.1]{RW-2016}; here the slope is
$i\le \frac23\kern1px g-\frac23$ and we recover
\cite[Thm.\,\textsc{e}]{Bonatto}. Combining both theorems, we obtain:

\begingroup
\def\thetheo{\ref{cor:A}}
\begin{cor}
  Let $\theta\colon L\to\on{BO}(2n)$ be a spherical tangential structure with
  $n$-connected $L$. If $2n\ge 6$ (or $2n=2$ and $\theta$ is admissible to
  \cite[Thm.\,7.1]{RW-2016}), we have, for each path-connected
  $\dE[\theta]{2n}$-algebra $A$ and each $g\ge 0$, isomorphisms
  \[H_i\bigl(\dW[\theta]{g,1}[A]\bigr) \cong
    H_i\bigl(\Omega_0^\infty\on{MT}\theta\times
    \Omega^\infty\bigl(\bigl(\Sigma^{\infty-2n}\B^{2n} U A\bigr)_\h{\Omega
      L}\bigr)\bigr)\]%
  for every $i$ small enough compared to $g$ to satisfy the conditions of \cref{thm:A}.
\end{cor}
\endgroup

We also discuss several special cases: If the $\dE[\theta]{2n}$-action on $A$ is
only partially defined, then one defines the factorisation homology
$\int^\theta_W A$ by first completing $A$ to an honest
$\dE[\theta]{2n}$-algebra, see \cref{constr:Completion}, and then applying the
above definition. For example, if $M\subseteq\dE[\theta]{2n}(1)$ is a submonoid
in the monoid of unary operations, then each based $M$-space $X$ can be regarded
as a partial $\dE[\theta]{2n}$-algebra. In this case, we get the following:

\begingroup
\def\thetheo{\ref{thm:C}}
\begin{thm}
  Let $\theta\colon L\to \on{BO}(2n)$ be a spherical tangential structure with
  $n$-connected $L$. If $M\subseteq \dE[\theta]{2n}(1)$ is a well-pointed
  submonoid and $X$ is a based $M$-space, regarded as a partial
  $\dE[\theta]{2n}$-algebra, then we have a weak equivalence of loop
  spaces\looseness-1
  \[\textstyle\Omega\B \dW[\theta]{*,1}[X]
    \simeq \Omega^\infty\MT\theta\times\Omega^\infty\Sigma^\infty(X_\h{M}).\]
\end{thm}
\endgroup

Here $X_\h{M}$ is the based homotopy quotient. If $\theta$ is of the form
$\B G\to \on{BO}(2n)$ for some $(n-1)$-connected subgroup $G\subset \on{O}(2n)$,
then $\dE[\theta]{2n}(1)$ contains a submonoid equivalent to $G$, see
\cref{ex:BG}.  In the case of surfaces, $\theta$ being orientations, and $M$
being equivalent to $\on{SO}(2)$, \cref{thm:C} hence recovers
\cite[Cor.\,\textsc{d}]{Bonatto}. In the case of arbitrary dimensions and $X$
being $S^0$ with the trivial $\dE[\theta]{2n}(1)$-action, \cref{thm:C} becomes a
(puncture-stabilised) variation of \cite[Thm.\,\textsc{a}]{Bonatto-2020}, see
\cref{ex:punct}.\looseness-1

Finally, we deduce a description for a stabilised version of the aforementioned
moduli spaces $\caM^\theta_\partial(W,\ell_W)\ula{X}$ of $\theta$-framed moduli
spaces together with a map to a retractive space $X$ over $L$, whose fibre over
$b_0\in L$ we denote by $X_{\smash{b_0}}$:

\begingroup
\def\thetheo{\ref{cor:e}}
\begin{cor}
  Let $\theta\colon L\to\on{BO}(2n)$ be a spherical tangential structure with
  $n$-connected $L$. Moreover, let $X$ be a $2n$-connective retractive space
  over $L$.  Then we have a weak equivalence
  \[\on{hocolim}_{g\to \infty}\bigl(\caM^\theta_\partial(W_{g,1},\ell_{g,1})\ula{X}\bigr)^+\simeq\Omega^\infty_0\MT\theta\times\Omega^\infty\bigl((\Sigma^{\infty-2n}X_{b_0})_\h{\Omega L}\bigr).\]
\end{cor}
\endgroup

\subsection{Strategy}

For the most part of this paper, we work in the model category of (compactly
generated) spaces, and consider algebras over topological operads. Only the last
section is written in the language of $\infty$-categories, after translating the
leftover question into this setting.\looseness-1

First, we recall from \cite{Tillmann-2000,BBPTY} the generalised surface operad
$\caW^\theta$, whose operation spaces are mo\-du\-li spaces of $\theta$-framed
manifolds of type $W_{g,1}$ for varying $g$, together with several embedded
discs that serve as inputs. We then establish a zig-zag\linebreak
$\smash{\dE[\theta]{2n}\stackrel{\kappa}{\leftarrow}\dF[\theta]{2n}\to
  \caW^\theta}$ of operad maps, where $\kappa$ is an equivalence. Very
generally, for each operad map $\rho\colon \caP\to \caO$, the forgetful functor
$\rho^*$ from $\caO$-algebras to $\caP$-algebras admits a (derived) left-adjoint
$\caO\otimes^\bbL_\caP({-})$, called \emph{pushforward}.  For any
$\dF[\theta]{2n}$-algebra $B$, the pushforward
$\caW^\theta\otimes^\bbL_{\dF[\theta]{2n}} B$ has a canonical grading by genus,
and we show the following:\looseness-1

\begingroup
\def\thetheo{\ref{prop:A}}
\begin{prop}
  Let $\theta\colon L\to \on{BO}(2n)$ be a spherical tangential structure. Then
  we have, for each $\dE[\theta]{2n}$-algebra $A$, a graded equivalence
  $\smash{\dW[\theta]{*,1}[A]\simeq \caW^\theta\otimes^\bbL_{\dF[\theta]{2n}}
    \kappa^*A}$.
\end{prop}
\endgroup

We point out that, instead of constructing the stabilisation maps and the
$A_\infty$-algebra structure on $\dW[\theta]{*,1}[A]$ ‘by hand’, we use this
equivalence and the fact that the right side canonically carries this structure.

We thus are left to understand the group-completion of pushforwards to
$\caW^\theta$-algebras. Here we use that if $L$ is $n$-connected, then
$\caW^\theta$ is an operad with homological stability in the sense of
\cite{BBPTY}. For such operads $\caO$, a description of the group-completion
$\Omega\B(\caO\otimes_\caP^\bbL A)$ has been established in
\cite{BBPTY,Bianchi-Kranhold-Reinhold} in the case where $\caP=E_0$ (the operad
of based spaces) and where $\caP$ is the operad of based $G$-spaces for some
topological group $G$ mapping to $\caO(1)$. Both results are based on the
observation that in the derived setting, $\dE\infty$ is the terminal operad, and
hence each operad $\caO$ has an essentially unique operad map to $\dE\infty$. We
show the following generalisation:\looseness-1

\begingroup
\def\thetheo{\ref{prop:B}}
\begin{prop}
  Let $\caO$ be an operad with homological stability, let $\caP$ be a proper
  operad with $\caP(0)\simeq *$, and let $\caP\to \caO$ be a map of operads
  under $E_0$. Then the map of $\caO$-algebras
  \[(\caO\otimes^\bbL_\caP A)\to (\caO\otimes^\bbL_\caP *) \times (\dE\infty
    \otimes^\bbL_\caP A)\]%
  that is comprised of $A\to *$ and $\caO\to E_\infty$ induces an equivalence on
  group-completions.
\end{prop}
\endgroup

Here, properness is a mild point-set topological assumption, see \cref{defi:opd}
for details. In the case of $\caP=\dF[\theta]{2n}$ and $\caO=\caW^\theta$, the
group-completion of the first factor is equivalent to the infinite loop space
associated with the tangential Thom spectrum $\on{MT}\theta$
\cite{Madsen-Weiss,Galatius-RW-2014}. Then \cref{prop:A,prop:B}, together with
the fact that the counit
$\dE[\theta]{2n}\otimes^\bbL_{\dF[\theta]{2n}} \kappa^* A\to A$ is an
equivalence, show:\looseness-1

\begingroup
\def\thetheo{\ref{cor:OHStang}}
\begin{cor}
  Let $\theta\colon L\to\on{BO}(2n)$ be a spherical tangential structure with
  $n$-connected $L$, and let $A$ be an $\dE[\theta]{2n}$-algebra. Then we have a weak
  equivalence of loop spaces\looseness-1
  \[\Omega\B\dW[\theta]{*,1}[A] \simeq \Omega_0^\infty\MT\theta
    \times \Omega\B(\dE\infty \otimes^\bbL_{\dE[\theta]{2n}} A).\]
\end{cor}
\endgroup

This reduces the original question to understanding the group-completion of the
$\dE\infty$-algebra $\dE\infty\otimes^\bbL_{\dE[\theta]{d}} A$ for a given
$\dE[\theta]{d}$-algebra $A$.  We start by discussing several special cases in
\cref{sec:smaller}. %
A general answer is established in
\cref{sec:eqBar}:\looseness-1

\begingroup
\def\thetheo{\ref{prop:C}}
\begin{prop}
  Let $\theta\colon L\to\B\on{O}(d)$ be a tangential structure with connected
  $L$ and let $A$ be an $\dE[\theta]d$-algebra. Then the shifted suspension
  spectrum $\Sigma^{\infty-d}\B^d UA$ carries an $E_1$-action by the loop space
  $\Omega L$ and we have an equivalence of connective spectra
  \[\B^\infty(E_\infty\otimes^\bbL_{\dE[\theta]d} A)\simeq
  \bigl(\Sigma^{\infty-d}\B^d UA\bigr)_\h{\Omega L}.\]
\end{prop}
\endgroup

The main result, \cref{thm:B}, is obtained by combining these
propositions.

\subsection*{Acknowledgements}
First, I would like to thank Luciana Basualdo Bonatto who was always open to
discuss the ideas I developed by looking at her work from a different
perspective.  Second, I am grateful to Manuel Krannich, who helped me on
numerous occasions with various technical details, and who also suggested to
describe the general pushforward with $\infty$-categorical methods.
Third, I owe thanks to Andrea Bianchi for several illuminating discussions on
the subject. Finally, I thank the anonymous referee for several illuminating
comments, which helped improve the article significantly.\looseness-1

\section{Basic notions}
\subsection{Monads and operads}

\begin{defi}
  By a \emph{space}, we mean a compactly generated topological space.  Limits
  are taken in the category $\Top$ of compactly generated topological spaces.
\end{defi}

\begin{defi}
  Let $\caC$ be a category and let $T$ be a monad in $\caC$. We denote the
  forgetful functor from $T$-algebras back to $\caC$ by $U^T$. If
  $S\colon \caC\to\caC'$ is a right $T$-functor, then each $T$-algebra $A$ gives
  rise to a simplicial object in $\caC'$, which we call the \emph{two-sided bar
    construction}, given by
  $\B_\bullet(S,T,A)\coloneqq \left([p]\mapsto
      ST^pU^TA\right){}_{p\in\Del^\op}$.
\end{defi}

\begin{expl}\label{ex:hQuot}
  For each topological group $G$, the assignment $\bbG(X)=G\times X$ is a monad
  $\bbG$ in $\Top$, and $\bbG$-algebras are the same as $G$-spaces.  Moreover,
  the identity functor $\mathbb{1}$ is a right $\bbG$-functor by projecting
  $\bbG(X)$ to the second factor. If $G$ is well-pointed and $X$ is a $G$-space,
  then $|\B_\bullet(\mathbb{1},\bbG,X)|$ is our preferred model for the homotopy
  quotient $X\sslash G$.\looseness-1
\end{expl}

\begin{lem}\label{lem:quotFirst}
  Let $T$ be a monad in $\caC$, $A$ a $T$-algebra, $G$ a well-pointed topological
  group, and $S\colon \caC\to\Top^G$ a right $T$-functor. Then
  $S\sslash G\colon \caC\to \Top$ is a right $T$-functor and
  \[\lvert \B_\bullet(S\sslash G,T,A)\rvert \cong \lvert
    \B_\bullet(S,T,A)\sslash G\rvert.\]
\end{lem}
\begin{proof}
  This follows from the fact that the two spaces in question are the two
  possibilities of realising the bisimplicial space
  $\B_\bullet(\B_\circ(\bEn,\bbG,S),T,A)=\B_\circ(\bEn,\bbG,\B_\bullet(S,T,A))$.
\end{proof}

\begin{defi}\label{defi:opd}
  By an \emph{operad} $\caO$, we mean a symmetric, monochromatic operad in
  spaces, see \cite[Def.\,1.1]{May-1972}, and we additionally require that all
  operation spaces $\caO(r)$ are Hausdorff. The identity operation is denoted by
  $\mathbf{1}=\mathbf{1}_\caO$.  An \emph{equivalence} $\caP\to \caO$ is an
  operad map such that all $\caP(r)\to \caO(r)$ are weak equivalences of
  spaces.\looseness-1

  We say that an operad $\caO$ is \emph{$\frS$-free} if for each $r\ge 0$, the
  $\frS_r$-action on $\caO(r)$ is free, and we call $\caO$ \emph{well-pointed}
  if the inclusion $\{\mathbf{1}_\caO\}\hookrightarrow \caO(1)$ is a Hurewicz
  cofibration. Finally, we call $\caO$ \emph{proper} if it is $\frS$-free and
  well-pointed.
\end{defi}

\begin{expl}\label{ex:Ed}
  We consider the \emph{$d$-discs operad} $\dE d$ where $\dE d(r)$ is the space
  of embeddings $\ul{r}\times D^d\hookrightarrow D^d$ (with
  $\ul{r}\coloneqq \{1,\dotsc,r\}$), which are on each disc of the form
  $z\mapsto \dot z + \rho_i\cdot z$ for some fixed $\dot z\in D^d$ and
  $\rho>0$. We have an inclusion $\dE d\hookrightarrow \dE{d+1}$ by extending
  the above description along
  $D^d=D^d\times\{0\}\hookrightarrow D^{d+1}$, and
  we call the colimit $\dE\infty$.

  For each $0\le d\le\infty$, we have $E_d(0)=*$ and the operad $E_d$ is
  proper. Moreover, the operation spaces $E_\infty(r)$ are contractible for each
  $r\ge 0$.\looseness-1
\end{expl}

\begin{defi}
  An \emph{$\caO$-algebra} is a space $A$, together with maps
  $\caO(r)\times_{\frS_r} A^r\to A$ that are associative and unital
  \cite[Def.\,1.2]{May-1972}. A map of $\caO$-algebras is called
  \emph{equivalence} if it is a weak equivalence on underlying spaces.
\end{defi}

\begin{expl}
  For each operad $\caO$, the space $\caO(0)$ of arity-$0$ operations
  is itself an $\caO$-algebra; it is actually the \emph{initial} $\caO$-algebra.
\end{expl}

\begin{defi}\label{defi:freeAlg}
  We denote by $U^\caO$ the forgetful functor from $\caO$-algebras to spaces.
  Its left-adjoint $F^\caO$ is given by taking $X$ to
  $\coprod_r \caO(r)\times_{\frS_r}X^r$, the $\caO$-action induced by the
  composition inside $\caO$. We denote the monad of this adjunction by $\bbO$;
  and in general, the monad associated to an operad gets the same letter in
  blackboard bold. Note that algebras over the monad $\bbO$ are the same as
  algebras over the operad $\caO$.
\end{defi}

\begin{defi}
  If $\rho\colon\caP\to\caO$ is a map of operads, then the forgetful functor
  $\rho^*$ from $\caO$-algebras to $\caP$-algebras has a left-adjoint, given by
  taking the free $\caO$-algebra and quotienting out all relations from the
  existing $\caP$-action. We call this left-adjoint \emph{pushforward} and
  denote it by $\caO\otimes_\caP ({-})$. Note that
  $\caO\otimes_\caP F^\caP\cong F^\caO$.
\end{defi}

Under mild point-set topological assumptions, there is a homotopy-invariant
replacement for the $\caO\otimes_\caP-$ that has a convenient simplicial
description:

\begin{rem}\label{rem:freeRes}
  If $\caP$ and $\caO$ are $\frS$-free, then their categories of algebras carry
  a model structure \cite{Berger-Moerdijk-2003} and $\caO\otimes_\caP({-})$ is a
  left Quillen functor, whose left-derivation we denote by
  $\caO\otimes^\bbL_\caP ({-})$.  If $\caP$ is well-pointed, then we have the
  following explicit model: The augmented simplicial $\caP$-algebra
  $\B_\bullet(F^\caP,\bbP,A)\to A$ is proper and has an extra degeneracy, and
  hence is a $\bbP$-free simplicial resolution of $A$ in the sense of
  \cite[Def.\,8.18]{Galatius-Kupers-RW-2018}.
  Using $U^\caO(\caO\otimes_\caP F^\caP)=\bO$, a model for
  $\caO\otimes^\bbL_\caP A$ is given by%
  \[|U^\caO \B_\bullet(\caO\otimes_\caP F^\caP,\bbP,A)| =
    |\B_\bullet(\bbO,\bbP,A)|,\]%
  together with the $\bbO$-action
  $\bbO|\B_\bullet(\bbO,\bbP,A)|\cong |\B_\bullet(\bbO^2,\bbP,A)|\to
  |\B_\bullet(\bbO,\bbP,A)|$, using that $\bbO$ commutes with geometric
  realisations \cite[Lem.\,9.7]{May-1972}. This description is functorial in
  $A$, and will be our preferred model throughout the article.

  If $\caQ$ is a third proper operad, together with a map $\caQ\to \caP$, then
  we have a natural weak equivalence
  $\caO\otimes^\bbL_\caP (\caP\otimes^\bbL_\caQ ({-})) \to \caO\otimes_\caQ^\bbL
  ({-})$. An elementary simplicial argument for this fact has been spelled out
  in the proof of \cite[Lem.\,5.12]{Bianchi-Kranhold-Reinhold}.\looseness-1
\end{rem}

For the following lemma, recall that for an operad $\caP$, an operad
\emph{under} $\caP$ is an operad $\caO$ that comes with a preferred operad map
$\caP\to \caO$. A map $\caO\to \caO'$ between operads under $\caP$ is required
to make the obvious triangle commute.

\begin{lem}\label{lem:Einfty}
  Let $\caP$ be a proper operad.
  \begin{enumerate}
  \item If $\caO$ is $\frS$-free operad under $\caP$ and $A\to A'$ is a map of
    $\caP$-algebras, then the induced map
    $\caO\otimes^\bbL_\caP A\to \caO\otimes^\bbL_\caP A'$ is an equivalence of
    $\caO$-algebras.
  \item If $A$ is a $\caP$-algebra and $\rho\colon \caO\to \caO'$ is an equivalence of
    $\frS$-free operads under $\caP$, then
    $\caO\otimes^\bbL_\caP A\to \rho^*(\caO'\otimes^\bbL_\caP A)$ is an equivalence of
    $\caO$-algebras.
  \end{enumerate}
\end{lem}

By abuse of notation, we will occasionally skip the symbol $\rho^*$ when it is
clear from the context that only the underlying $\caO$-algebra structure is
taken into account.

\begin{proof}
  As each $\caO(r)$ is Hausdorff and $\frS_r$ acts freely on $\caO(r)$, the map
  $\caO(r)\to \caO(r)/\frS_r$ is a covering. We hence get, for each $X$, a fibre
  sequence $X^r\to \caO(r)\times_{\frS_r}X^r\to \caO(r)/\frS_r$ which is natural
  in $X$.  By the five lemma, this shows that the monad $\bbO$ preserves weak
  equivalences among arbitrary spaces. The same argument applies to $\bbP$, and
  therefore, the simplicial map $\B_\bullet(\bbO,\bbP,A\to A')$ is a levelwise
  equivalence.  Similarly, the map $\bbO X\to \bbO' X$ is a weak equivalence for
  each space $X$, see \cite[Lem.\,11]{Kupers-Miller} for details, whence also
  $\B_\bullet(\bbO\to\bbO',\bbP,A)$ is a levelwise equivalence.
  
  Finally, since $\caO$ and $\caO'$ are $\frS$-free and $\caP$
  is proper, all involved simplicial spaces are proper.  This shows that the
  maps induced on realisations are weak equivalences.%
\end{proof}

\begin{constr}\label{constr:Einf}
  Let $\caP$ be a proper operad. Then the operad $\caP\times E_\infty$ is again
  proper, the projection to the first factor
  $\pi_1\colon \caP\times E_\infty \to\caP$ is an equivalence of proper operads,
  and the projection to the second factor is an operad map
  $\pi_2\colon \caP\times E_\infty\to E_\infty$.

  Let $A$ be a $\caP$-algebra. If
  $\smash{\caP\stackrel{\nu_1}{\leftarrow}\caQ\stackrel{\nu_2}{\to} E_\infty}$
  is any other zig-zag of proper operads such that $\nu$ is an equivalence, then
  we have an equivalence
  $E_\infty\otimes^\bbL_{\caP\times E_\infty}\pi^*_1 A\simeq
  E_\infty\otimes^\bbL_\caQ \nu^*_1 A$ of $E_\infty$-algebras.
  We will denote this homotopy type simply by $E_\infty\otimes^\bbL_\caP A$,
  noting that in the case where $\caP$ already comes with a map $\rho$ to
  $E_\infty$, the case of $\nu_1=\on{id}_\caP$ and $\nu_2=\rho$
  shows that the two competing definitions agree up to equivalence.
\end{constr}

\begin{constr}\label{constr:Ainft}
  Let $\mathrm{Ass}$ denote the associative operad; its algebras are topological
  monoids. An \emph{\emph{$A_\infty$}-operad} is a proper operad $\caA$ with an
  equivalence $\caA\to \mathrm{Ass}$.
  If $A$ is an \emph{$A_\infty$-algebra}, i.e.\ an algebra over some
  $A_\infty$-operad $\caA$, then it admits a \emph{bar construction}
  $\B A\coloneqq \B(\mathrm{Ass}\otimes^\bbL_\caA A)$, and hence a
  \emph{group-completion} $\Omega \B A$.

  Via the map $\pi_2\colon\on{Ass}\times E_\infty\to E_\infty$, each
  $E_\infty$-algebra $A$ is an algebra over the $A_\infty$-operad
  $\on{Ass}\times E_\infty$. If $\rho\colon\caA\to E_\infty$ is another map from
  an $A_\infty$-operad to $E_\infty$, then the monoids
  $\mathrm{Ass}\otimes^\bbL_{\on{Ass}\times E_\infty} \pi_2^*A$ and
  $\mathrm{Ass}\otimes^\bbL_{\caA} \rho^*A$ are equivalent, and so are their bar
  costructions.\looseness-1
\end{constr}

\subsection{Tangential structures and framed little discs}

\begin{defi}
  For two smooth manifolds $M$ and $N$, possibly with boundary, we denote by
  $\Emb(M,N)$ the space of smooth embeddings $M\hookrightarrow N$. Note that so
  far, we do not impose any boundary condition on the embeddings.
\end{defi}

\begin{defi}\label{defi:tang}
  A \emph{tangential structure} is a fibration $\theta\colon L\to \B\on{O}(d)$
  such that the total space $L$ is connected.  If $V_d$ denotes the universal
  vector bundle over $\B\on{O}(d)$, then a \emph{$\theta$-framing} on a smooth
  $d$-dimensional manifold $W$ is a bundle map $\ell_W\colon TW\to\theta^*V_d$.
  The space $\Fr^\theta(W)$ of all $\theta$-framings on $W$ is by definition
  $\Bun(TW,\theta^*V_d)$.
\end{defi}

\begin{constr}\label{constr:embth}
  Given $\theta$-framed manifolds $(W,\ell_W)$ and $(W',\ell_{W'})$, the space
  of \emph{$\theta$-framed embeddings
    $(W',\ell_{W'})\hookrightarrow (W,\ell_W)$} should model the homotopy fibre
  of the map $\Emb(W',W)\to \Fr^\theta(W')$ with
  $\alpha\mapsto \ell_{W}\circ T\alpha$ at $\ell_{W'}$.  In
  \cite[Def.\,17]{Kupers-Miller}, the authors give a ‘Moore path’ description
  as\looseness-1
  \[\Emb^\theta(W',W)\!\coloneqq\!\left\{\hspace*{-5px}
         \begin{array}{l}
           (\alpha,t^\bullet,\gamma)\in \on{Emb}(W',W)\times [0,\infty)^{\pi_0(W')}
           \times \Fr^\theta(W')^{[0,\infty)}\\
           \text{$\gamma(0)=\ell_{W'}$ and $\gamma|_{[t^i,\infty)}\equiv \ell_{W}\circ T\alpha$
           on each $i\in\pi_0(W')$.}
         \end{array}\hspace*{-5px}
       \right\}.\] It admits a strict composition
     $\Emb^\theta(W',W)\times\Emb^\theta(W'',W')\to \Emb^\theta(W'',W)$ by
     setting
     $(\alpha,t,\gamma)\circ (\alpha',t',\gamma') =
     (\alpha\circ\alpha',t'+t,\bar\gamma)$ on each path-component, with
  \[\bar\gamma(s)\coloneqq\begin{cases}
      \gamma'(s) & \text{for $0\le s\le t'$,}\\
      \gamma(s-t')\circ T\alpha' & \text{for $t'\le s\le t'+t$,}\\
      \ell_{W}\circ T\alpha\circ T\alpha' & \text{for $t'+t\le s$,}
    \end{cases}\] %
  and since we treated different components of $W'$ separately, we can also take
  disjoint unions of $\theta$-framed embeddings, resulting in a family of maps
  \[\Emb^\theta(W'_1,W_1)\times \Emb^\theta(W'_2,W_2)\to \Emb^\theta(W_1'\sqcup W_2',W_1\sqcup W_2).\]
  This constitutes a topologically enriched symmetric monoidal category with
  objects being $\theta$-framed manifolds, and morphisms being $\theta$-framed
  embeddings  \cite[Def.\,20]{Kupers-Miller}.

  We have a maps $\Emb^\theta(W',W)\to \Emb(W',W)$ compatible with compositions
  and disjoint unions. In the case where $\theta$ is the identity on
  $\B\on{O}(d)$, these maps are equivalences.
\end{constr}

\begin{nota}
  We often suppress the length of the Moore path in the tuple and just write
  $(\alpha,\gamma)$, meaning that for each $i\in \pi_0(W')$, with
  $W'_i\subseteq W'$ being the corresponding path-component, we have a path
  $\gamma^i\colon [0,t^i]\to \Fr^\theta(W'_i)$.\looseness-1
\end{nota}

\begin{defi}
  We fix, once and for all, a $\theta$-framing $\ell_{\R^d}$ of $\R^d$ which
  under the tautological trivialisation $\R^d\times\R^d\cong T\R^d$ only depends
  on the fibre factor (this is the same as choosing a basepoint $b_0\in L$ and
  parametrising the fibre $\theta^*V_d|_{\smash{b_0}}$).

  Let $\ell_{0,1}$ be the restriction of $\ell_{\R^d}$ to $D^d$ (the index
  stands for ‘genus $0$ and one boundary component’). Then we define the
  \emph{$\theta$-framed $d$-discs operad} with
  $\dE[\theta]d(r)\coloneqq \Emb^\theta(\ul{r}\times D^d,D^d)$, with
  $\ul{r}=\{1,\dotsc,r\}$ as before, where the operadic composition is given by
  composition of disjoint unions of embeddings.%
\end{defi}

\begin{expl}\label{ex:unfr}
  We have an operad map $\imath\colon E_d\to \dE[\theta]d$ by endowing each
  $\alpha\colon D^d\hookrightarrow D^d$ with $\alpha(z)= \dot z+r\cdot z$ with the path
  $[0,\log(\frac1r)]\to \Fr^\theta(D^d)$ taking $s$ to
  $e^{-s}\cdot \ell_{0,1}$.  If $\theta$ is the universal bundle
  $\on{E}\on{O}(d)\to \B\on{O}(d)$, this map is an equivalence of operads,
  leading to the usual defect in nomenclature that the classical (‘unframed’)
  operad $E_d$ is equivalent to the $E_d$-operad for the tangential structure of
  framings.
\end{expl}

\begin{lem}\label{lem:unaries}
  The space $\dE[\theta]d(1)$ of unary operations is equivalent to $\Omega L$.
\end{lem}
\begin{proof}
  The map $\Bun(TD^d,\theta^*V_d)\to L$ that only remembers the value
  of $0\in D^d$ is a fibration, whose fibre is homeomorphic to
  $\Bun_0(TD^d,TD^d)$, the space of bundle maps fixing
  $0\in D^d$.  Moreover, we have a zig-zag of maps\looseness-1
  \[
    \begin{tikzcd}[column sep=3em]
      \Emb(D^d,D^d)\ar{d}[swap]{\alpha\mapsto \ell_{0,1}\circ T\alpha} & \on{O}(d)\ar{r}{\omega\mapsto T\omega|_{\smash{D^d}}}\ar[d]\ar{l}[swap]{\omega|_{\smash{D^d}}\mapsfrom \omega} & \Bun_{0}(TD^d,TD^d)\ar{d}{\ell_{0,1}\circ ({-})}\\
      \Bun(TD^d,\theta^*V_d)\ar[r,equal] & \Bun(TD^d,\theta^*V_d)\ar[r,equal] & \Bun(TD^d,\theta^*V_d),
    \end{tikzcd}
  \]
  where the horizontal maps are equivalences. This induces a zig-zag of
  equivalences among homotopy fibres. They, in turn, are
  $\Emb^\theta(D^d,D^d)=\dE[\theta]d(1)$ and
  $\Omega L$.
\end{proof}

\subsection{Factorisation homology and decorated moduli spaces}
\label{subsec:FH}

We start by repeating the model from \cite[Def.\,34+43]{Kupers-Miller}:

\begin{rmd}
  Let $(W,\ell_W)$ be a $\theta$-framed manifold and let $A$ be a
  $\dE[\theta]d$-algebra.  We define the \emph{factorisation homology}
  $\int^\theta_W A\coloneqq |\B_\bullet(\dbE[\theta]W,\dbE[\theta]d,A)|$, where
  $\dbE[\theta]W$ is the right $\dbE[\theta]d$-functor that takes a space $X$ to
  $\coprod_{r\ge 0}\Emb^\theta(\ul{r}\times D^d,W)\times_{\frS_r} X^r$.
\end{rmd}

Informally, $\int^\theta_W A$ is the space of configurations of discs inside
$W$, each disc carrying a label in $A$, and the $\dE[\theta]d$-action on $A$ is
balanced with precomposing embeddings of discs; we call such a datum a
‘decoration’ of $W$.  We want to consider the \emph{moduli space} of such
decorated manifolds.
As done in \cite{Galatius-RW-2017,Galatius-RW-2018} for moduli spaces of
manifolds without decorations, this can be implemented by extending
$\smash{\int_W^\theta A}$ so that the $\theta$-framing of $W$ is allowed to
vary, and then quotient out the action of the usual topological group
$\Diff_\partial(W)$.  We make precise what we mean by this:\looseness-1

\begin{constr}
  Let $(W,\ell_W)$ be a $\theta$-framed manifold, possibly with boundary.  We
  fix a collar of $\partial W$ and let
  $\Fr^\theta_\partial(W)\subseteq \Fr^\theta(W)$ be the space of
  $\theta$-framings $\ell$ whose restriction to that collar agrees with the
  restriction of $\ell_W$. Moreover, let
  $\Fr^\theta_\partial(W,\ell_W)\subseteq\Fr^\theta_\partial(W)$ be the subspace
  containing those path-components intersecting the $\Diff_\partial(W)$-orbit of
  $\ell_W$.

  If $(W',\ell_{W'})$ is another $\theta$-framed manifold, then we want the space
  $\ul\Emb^\theta(W',W)$ to model the homotopy fibre of\looseness-1
  \[\Emb(W',W)\times \Fr^\theta_\partial(W,\ell_W)\to
    \Fr^\theta(W'),\quad (\alpha,\ell)\mapsto \ell\circ T\alpha\]%
  at $\ell_{W'}$. Similar to \cref{constr:embth}, this is achieved by defining
  $\ul\Emb^\theta(W',W)$ as the subspace of
  $\Emb(W',W)\times [0,\infty)^{\pi_0(W')}\times
  \Fr^\theta(W')^{[0,\infty)}\times \Fr^\theta_\partial(W,\ell_W)$ containing
  all tuples $(\alpha,t,\gamma,\ell)$ such that $(\alpha,t,\gamma)$ is a
  $\theta$-framed embedding $(W',\ell_{W'})\hookrightarrow (W,\ell)$. We then
  have compositions
  $\ul\Emb^\theta(W',W)\times \Emb^\theta(W'',W')\to \ul\Emb^\theta(W'',W)$ that
  are given by
  \[(\alpha,t,\gamma,\ell)\circ (\alpha',t',\gamma')=((\alpha,t,\gamma)\circ
  (\alpha',t',\gamma'),\ell).\]
\end{constr}

\begin{constr}\label{constr:mod}
  The topological group $\Diff_\partial(W)$ of diffeomorphisms of $W$ that
  preserve the collar of $W$ acts on
  $\ul\Emb^\theta(W',W)$ via
  $\phi\cdot (\alpha,t,\gamma,\ell)=(\phi\circ\alpha,t,\gamma,\ell\circ
  T\phi^{-1})$ and we denote the homotopy quotient by
  \[\caM^\theta_\partial(W,\ell_W)^{(W'\kern-1px,\kern1px\ell_{W'}\kern-.5px)}\coloneqq
    \ul\Emb^\theta(W',W)\sslash \Diff_\partial(W).\]%
\end{constr}

\begin{expl}
  The embedding space $\ul\Emb^\theta(\emptyset,W)$ is the same as
  $\Fr^\theta_\partial(W,\ell_W)$, and so
  $\caM^\theta_\partial(W,\ell_W)^{\emptyset} =
  \Fr^\theta_\partial(W,\ell_W)\sslash \Diff_\partial(W)$, which agrees with the
  classical description of the moduli space of $(W,\ell_W)$ as in
  \cite{Galatius-RW-2018}; hence the notation.

  We point out that in general,
  $\caM^\theta_\partial(W,\ell_W)^{(W'\kern-1px,\kern1px\ell_{W'})}$ need not be
  path-connected, but in the special case of $(W',\ell_{W'})=\emptyset$, it is
  path-connnected by construction.
\end{expl}

\begin{constr}\label{constr:modDiff}
  Let $(W,\ell_W)$ be a $\theta$-framed manifold. We have a right
  $\dbE[\theta]d$-functor from spaces to $\Diff_\partial(W)$-spaces by\looseness-1
  \[\udbE[\theta]W(X)\coloneqq \coprod_{r\ge 0}\ul\Emb^\theta(\ul{r}\times D^d,W)\times_{\frS_r}X^r,\]
  where the transformation
  $\udbE[\theta]W\circ \dbE[\theta]d\to \udbE[\theta]W$ is
  induced by the above composition, using that precomposition is
  $\Diff_\partial(W)$-equivariant. When passing to homotopy quotients,
  we get a right $\dbE[\theta]d$-functor from spaces to spaces by
  $\mdE[\theta]W\coloneqq \udbE[\theta]W\sslash \Diff_\partial(W)$,
  and we define the \emph{moduli space of manifolds of type $W$ with decorations in $A$}
  as
  \[W^\theta[A]\coloneqq |\B_\bullet(\mdE[\theta]W,\dbE[\theta]d,A)|
    \cong |\B_\bullet(\udbE[\theta]W,\dbE[\theta]d,A)|\sslash
    \Diff_\partial(W).\]%
\end{constr}

\begin{expl}\label{ex:point}
  The one-point space $*$, together with its unique $\dE[\theta]d$-algebra
  structure, is the free $\dE[\theta]d$-algebra over the empty space
  $\emptyset$. We hence have an augmentation
  \[\B_\bullet(\mdE[\theta]W,\dbE[\theta]d,*)=
    \B_\bullet(\mdE[\theta]W,\dbE[\theta]d,\dbE[\theta]d(\emptyset))\to
    \mdE[\theta]W(\emptyset)=\caM^\theta_\partial(W,\ell_W),\]%
  induced by the transformation
  $\mdE[\theta]W\circ \dbE[\theta]d\to
  \mdE[\theta]W$, and this augmentation admits an extra degeneracy
  induced by the unit of the monad $\dbE[\theta]d$ in the last argument. This
  shows that after geometric realisation, we obtain an equivalence\looseness-1
  \[W^\theta[*]=|\B_\bullet(\mdE[\theta]W,\dbE[\theta]d,*)|\stackrel{\simeq}{\longrightarrow}
    \caM^\theta_\partial(W,\ell_W).\]%
\end{expl}

\begin{rmk}\label{rmk:hRet}
  The assignment $A\mapsto W^\theta[A]$ is functorial in $\dE[\theta]d$-algebras
  and $*$ is both initial and terminal in $\dE[\theta]d$-algebras. It follows
  that $W^\theta[*]$ is a retract of $W^\theta[A]$ for each
  $\dE[\theta]d$-algebra $A$. In combination with \cref{ex:point}, we can
  conclude that each $W^\theta[A]$ contains $\caM^\theta_\partial(W,\ell_W)$ as
  a homotopy retract.
\end{rmk}

\begin{lem}\label{lem:Actd}
  If $A$ is path-connected, then $W^\theta[A]$ is path-connected as well.
\end{lem}
\begin{proof}
  We show that the aforementioned retract
  $W^\theta[*]\hookrightarrow W^\theta[A]$ is $0$-connected (i.e.\ surjective on
  $\pi_0$); then the statement follows from the fact that
  $W^\theta[*]\simeq \caM^\theta_\partial(W,\ell_W)$ is path-connected. Since the
  map in question is the geometric realisation of the simplicial map
  $\imath_\bullet\colon
  \B_\bullet(\mdE[\theta]W,\dbE[\theta]d,*)\to\B_\bullet(\mdE[\theta]W,\dbE[\theta]d,A)$,
  it suffices to check that the map $\imath_0$ among $0$-simplices is
  $0$-connected. The map $\imath_0$, however, is explicitly given by the union
  \[\coprod_{r\ge 0}\on{id}_{\ul\Emb^\theta(\ul r\times D^d,W)}\times_{\frS_r} (*\to A)^r,\]%
  which clearly is $0$-connected if $A$ is path-connected.
\end{proof}

\begin{expl}
  In the case where $\theta\colon \B \SO(2)\to \B\on{O}(2)$ is orientations of
  surfaces, $A$ is an $\dE[\theta]2$-algebra, and $W$ is an oriented surface,
  $W^\theta[A]$ is a ‘disc model’ for the generalised configuration space that
  has been described in \cite[§\,4]{Bonatto}. 
\end{expl}

We will see further special cases in \cref{sec:smaller}, e.g.\ moduli spaces
$\caM^{\theta,r}_\partial(W,\ell_W)$ of manifolds with $r$ permutable punctures
studied in \cite{CFB-Tillmann-2001,Bonatto-2020}, see \cref{ex:punct}.

Next, we establish a fibre sequence showing that $W^\theta[A]$ relates
factorisation homology $\int^\theta_W A$ and the classical moduli space
$\caM^\theta_\partial(W,\ell_W)$. This generalises the well-known fibre sequence
$\coprod_r C_r(\mathring W)\to \coprod_r\caM^{\theta,r}_\partial(W,\ell_W)\to
\caM^\theta_\partial(W,\ell_W)$, where $C_r(\mathring W)$ is the space of
unordered configurations of $r$ particles inside $\mathring W$:

\begin{prop}\label{prop:hofib}
  For each $\dE[\theta]d$-algebra $A$, we have a homotopy fibre sequence with a
  section
  \[\begin{tikzcd}
    \int_W^\theta A\ar[r] & W^\theta[A]\ar[r] &\caM^\theta_\partial(W,\ell_W).
  \end{tikzcd}\]
\end{prop}
\begin{proof}
  Consider the augmented proper simplicial space
  \[\B_\bullet(\udbE[\theta]W,\dbE[\theta]d,A)\to %
    \B_\bullet(\udbE[\theta]W,\dbE[\theta]d,*)\to %
    \dbE[\theta]W(\emptyset)=\Fr^\theta_\partial(W,\ell_W).\]%
  The augmentations
  $B_p(\udbE[\theta]W,\dbE[\theta]d,A)\to \Fr^\theta_\partial(W,\ell_W)$ are
  of the form $\udbE[\theta]W(Y)\to \Fr^\theta_\partial(W,\ell_W)$ taking a
  tuple
  $\smash{[\alpha,t,\gamma,\ell;y_1,\dotsc,y_r]\in \ul\Emb^\theta(\ul{r}\times
    D^d,W)\times_{\frS_r}Y^r}$ to $\ell$. These maps can
  easily be checked to be fibrations, in particular quasifibrations.
  The actual simplicial fibre of the augmentation is the proper simplicial space
  $\B_\bullet(\dbE[\theta]W,\dbE[\theta]d,A)$.
  By \cite[Lem.\,2.14]{Ebert-RW} (and the fact that all involved simplicial
  spaces are proper, enabling us to switch between think and thin realisations),
  it follows that the homotopy fibre of the realisation
  $|\B_\bullet(\udbE[\theta]W,\dbE[\theta]d,A)|\to
  \Fr^\theta_\partial(W,\ell_W)$ is equivalent to
  $|\B_\bullet(\dbE[\theta]W,\dbE[\theta]d,A)|=\int^\theta_A$, %
  and so we obtain a homotopy fibre sequence
  \[\textstyle \int^\theta_WA\to |\B_\bullet(\udbE[\theta]W,\dbE[\theta]d,A)|\to
    \Fr^\theta_\partial(W,\ell_W).\]%
  Now we observe that the second map in this sequence is
  $\Diff_\partial(W)$-equivariant, and a diagram chase shows that the induced
  map among homotopy quotients has the same homotopy fibre. This gives rise to
  the desried homotopy fibre sequence. Finally, the map
  $W^\theta[A]\to \caM^\theta_\partial(W,\ell_W)$ has a section as it is given
  by $W^\theta[A]\to W^\theta[*]$, followed by the equivalence
  $W^\theta[*]\to \caM^\theta_\partial(W,\ell_W)$ from \cref{ex:point}, and the
  first map as a section by the functoriality of $W^\theta[{-}]$, as mentioned
  before.
\end{proof}

\subsection{Generalised surfaces}
\label{subsec:GenSurf}

In order to be able to ‘stabilise’ the spaces $W^\theta[A]$, we restrict our
attention to a certain class of even-dimensional manifolds
$W_{g,1}= \#^g(S^n\times S^n)\setminus\mathring{D}^{2n}$, which we introduce in
this subsection. We will use a slightly different model for $W_{g,1}$ in order
to have more control over their $\theta$-framings, especially close to their
boundary.\looseness-1

\begin{defi}
  Let $n\ge 1$. Then we abbreviate
  $C\coloneqq D^{2n}\setminus \frac12 D^{2n}$. We consider the manifold
  $W_{0,1}\coloneqq D^{2n}$, together with the collar
  $\jmath_0\colon C\hookrightarrow D^{2n}$ and the $\theta$-framing
  $\ell_{0,1}$.

  Moreover, we consider the manifold
  $W_{1,1}\coloneqq (S^n\times S^n)\setminus \mathring{D}^{2n}$, where
  $\mathring{D}^{2n}$ is the interior of a disc inside $S^n\times S^n$, and fix,
  once and for all, a collar $\jmath_1\colon C\hookrightarrow W_{1,1}$.
\end{defi}

In order to prescribe a rather strict $\theta$-framing near the boundary of
$W_{1,1}$, we need the notion of spherical tangential structures from
\cite{Galatius-RW-2017}:\looseness-1

\begin{defi}
  Let $D^d\hookrightarrow S^d$ the inclusion of a hemisphere. We call a tangential
  structure $\theta\colon L\to \on{BO}(d)$ \emph{spherical} if each
  $\theta$-framing on $D^d$ can be extended on $S^d$.  As $L$ is assumed to be
  path-connected, this is equivalent to requiring that $S^d$ admits a
  $\theta$-structure.\looseness-1
\end{defi}

\begin{lem}\label{lem:collar}
  Let $\theta\colon L\to \on{BO}(2n)$ be spherical.  Then there is a
  $\theta$-framing $\ell_{1,1}$ on $W_{1,1}$ which is admissible in the sense of
  \cite[Def.\,1.3]{Galatius-RW-2018} and satisfies
  $\ell_{1,1}\circ T\jmath_1 = \ell_{0,1}|_{\smash{C}}$.
\end{lem}
\begin{proof}
  As $\theta$ is spherical, we can extend $\smash{\ell_{0,1}}$ to a
  $\theta$-framing $\ell$ of $S^{2n}$. Now we fix an embedding of the open disc
  $S^{2n}\setminus \frac12D^{2n}$ into the interior of $W_{1,1}$, and by using that
  $\ell|_{\smash{\mathring C}}$ factors through the trivial $\R^{2n}$-bundle
  over a point, we find an admissible $\theta$-framing of $W_{1,1}$ that
  restricts to $\ell$ on $S^{2n}\setminus\frac12D^{2n}$. By
  \cite[Lem.\,7.9]{Galatius-RW-2018}, we can extend this $\theta$-framing
  further onto the closed manifold $S^n\times S^n$. Finally, we remove
  $S^{2n}\setminus D^{2n}$ from $S^n\times S^n$ to obtain $W_{1,1}$ with the
  desired $\theta$-framing.\looseness-1
\end{proof}

\begin{constr}\label{constr:Wg1}
  For $g\ge 2$, we fix a rectilinear embedding
  $\alpha_g\colon \ul{\smash{g}}\times D^{2n}\hookrightarrow D^{2n}$ that
  satisfies
  $\alpha_g(\ul{\smash g}\times \frac12D^{2n})\subseteq \frac12D^{2n}$.  If we
  abbreviate
  $D^{\smash{2n}}_g\coloneqq D^{2n}\setminus \alpha_g(\ul{\smash g}\times
  \frac12 D^{2n})$, then $\alpha_g$ (co-)restricts to a map
  $\ul{\smash g}\times C\hookrightarrow D^{\smash{2n}}_g$, and we define the
  smooth manifold\looseness-1
  \[W_{g,1} \coloneqq (\ul{\smash g}\times W_{1,1})
    \cup_{\smash{\ul{\smash g}\times C}} D^{\smash{2n}}_g.\]
  We define $\ell_{g,1}\colon TW_{g,1}\to \theta^*V_{2n}$ to be
  $r^{-1}_i\cdot \ell_{1,1}$ on each $\{i\}\times TW_{1,1}$, where $r_i>0$ is
  the radius of the $i$\textsuperscript{th} disc, and to be $\ell_{0,1}$ on
  $D^{\smash{2n}}_g$, noting that these maps agree on the intersection of their
  domains. Then
  $\smash{\jmath_g\colon C\hookrightarrow D^{2n}_g\hookrightarrow W_{g,1}}$
  satisfies $\smash{\ell_{g,1}\circ T\jmath_g=\ell_{0,1}|_C}$.
\end{constr}

We thus have defined, for any $g\ge 0$, a $\theta$-framed manifold
$(W_{g,1},\ell_{g,1})$ with a collar $\jmath_g\colon C\to W_{g,1}$ of
the boundary. Accordingly, we require each
$\phi\in\Diff_\partial(W_{g,1})$ to satisfy $\phi\circ\jmath_g=\jmath_g$, and each
$\theta$-framing $\ell\in\Fr^\theta_\partial(W_{g,1},\ell_{g,1})$ to satisfy
$\ell\circ T\jmath_g=\ell_{g,1}\circ T\jmath_g$.

\begin{rmk}
  By picking $\theta$-framed embeddings
  $(W_{g,1},\ell_{g,1})\hookrightarrow (W_{g+1,1},\ell_{g+1,1})$, we can
  construct stabilisation maps $W^\theta_{g,1}[A]\to W^\theta_{g+1,1}[A]$. As it
  will turn out, these stabilisation maps are actually part of an
  $A_\infty$-structure on the disjoint union
  $\smash{W_{*,1}^\theta[A]\coloneqq \coprod_{g\ge 0}W_{g,1}[A]}$.

  Instead of making that formal, we will establish in \cref{prop:A} a graded
  equivalence between $W_{*,1}^\theta[A]$ and a space that canonically carries
  an $A_\infty$-structure.
\end{rmk}

\label{sec:2}

\section{Factorisation homology and the generalised surface operad}
\label{sec:3}
We fix a dimension $d=2n$ and a spherical tangential structure
$\theta\colon L\to \on{BO}(2n)$.  The goal of this section is to prove the
following statement, where $\caW^\theta$ is the $\theta$-framed generalised
surface operad (also called ‘manifold operad’) from
\cite{BBPTY}:%

\begin{prop}\label{prop:A}
  Let $\theta\colon L\to \on{BO}(2n)$ be a spherical tangential structure. Then
  we have, for each $\dE[\theta]{2n}$-algebra $A$, a graded equivalence
  $\smash{\dW[\theta]{*,1}[A]\simeq \caW^\theta\otimes^\bbL_{\dF[\theta]{2n}}
    \kappa^*A}$.
\end{prop}

Here $\dF[\theta]{2n}$ is an operad that comes with an operad map
$\dF[\theta]{2n}\to\caW^\theta$ and an equivalence
$\kappa\colon \dF[\theta]{2n}\to \dE[\theta]{2n}$ of operads, and the
pushforward $\caW^\theta\otimes^\bbL_{\dF[\theta]{2n}}({-})$ is graded
by genus. Each of these objects needs to be constructed first, and this is the
content of the next subsection.

\subsection{Geometric models for operads}

Informally, the generalised surface operad $\caW^\theta$ is associated to the
\acr{PROP} given by the subcategory of the $\theta$-framed bordism category,
containing as objects disjoint unions of $S^{2n-1}$ and as morphisms bordisms
diffeomorphic to unions of
$W_{g,r+1}=W_{1,1}\setminus (\ul{r}\times \mathring D^{2n})$, considered to have
$r$ incoming and one outgoing boundary components.  This description can be
found in an $\infty$-operadic setting in \cite[§\,6.2]{Krannich-Kupers}. In
particular, the operation space $\caW^\theta(r)$ is a model for the moduli space
$\caM^\theta_\partial(W_{g,r+1},\ell_{g,r+1})$.  When working with actual
topological operads, one uses a different model \cite{Tillmann-2000,BBPTY}, due
to the fact that the topological bordism category \cite{GMTW} is not strictly
monoidal.

We study yet a third model that uses spaces of submanifolds of $\R^\infty$, and I
thank the referee for showing a way to significantly simplify my original
technical construction. The main merit of this model is the fact that we can
make the aforementioned zig-zag
$\dE[\theta]{2n}\leftarrow \dF[\theta]{2n}\to\caW^\theta$ explicit. The
existence of such operad maps has been mentioned as a ‘folk theorem’ in
\cite[Rmk.\,6.15]{Horel}, but I am not aware of any reference for it.

\begin{constr}
  For each $d\ge 0$, we let $D^{d,\infty}\coloneqq D^d\times D^\infty$, and we
  identify $D^d$ with the subspace
  $D^d\times\{0\}^\infty\subset D^{d,\infty}$.

  As a slight variation of the operad $E_\infty$ from \cref{ex:Ed}, let
  $E_{d,\infty}(r)$ be the space of all embeddings
  $\beta\colon \ul{r}\times D^{d,\infty}\hookrightarrow D^{d,\infty}$ which are
  on each disc of the form $z\mapsto \dot z + \rho\cdot z$ for some
  $\dot z\in D^{d,\infty}$ and $\rho>0$, such that
  $\beta(\ul{r}\times \frac12D^d\times D^\infty)\subseteq \frac12D^d\times
  D^\infty$ holds. These spaces assemble into a proper operad $E_{d,\infty}$
  with contractible operation spaces.\footnote{The operad $E_{d,\infty}$ is
    actually equivalent to $E_\infty$, but we will not use this fact.}
\end{constr}

\begin{constr}
  Recall the annulus $C\coloneqq D^{2n}\setminus \frac12D^{2n}$, the manifolds
  $W_{g,1}$, and the collars $\jmath_g\colon C\hookrightarrow W_{g,1}$ from
  \cref{subsec:GenSurf}.  We define $\Emb_\partial(W_{g,1},D^{2n,\infty})$ as the space
  of smooth embeddings $\eta\colon W_{g,1}\hookrightarrow D^{2n,\infty}$ such
  that $\eta\circ \jmath_g$ agrees with the canonical embedding
  $C\hookrightarrow D^{2n,\infty}$ and the remainder
  $\eta(W_{g,1}\setminus \jmath_g(C))$ lies inside
  $\frac12 D^{2n}\times \mathring D^{2n,\infty}$.\looseness-1
\end{constr}

\begin{constr}\label{constr:tW}
  Let
  $\tilde\caW_g^\theta(r)\subset \Emb_\partial(W_{g,1},D^{2n,\infty})\times
  \Fr^\theta_\partial(W_{g,1},\ell_{g,1})\times E_{2n,\infty}(r)$ the subspace
  of triples $(\eta,\elll,\beta)$ such that
  $\eta(W_{g,1})\cap \beta(\ul{r}\times D^{2n,\infty})=\beta(\ul{r}\times
  D^{2n})$ holds and for the (co-)restriction
  $\beta|_{\ul{r}\times D^{2n}}\colon \ul{r}\times D^{2n}\to W$, we have a
  (strict) equality
  \[\elll\circ T\eta^{-1}\circ T\beta|_{\ul{r}\times D^{2n}} = \ul{r}\times
    \ell_{0,1}.\]%
  We have an action of $\Diff_\partial(W_{g,1})$ on $\tilde\caW_g^\theta(r)$ by
  precomposing $\eta$ with $\phi$ and $\elll$ with $T\phi$, and we call its
  (actual) quotient $\caW^\theta_g(r)$.  Then elements in $\caW^\theta_g(r)$ are
  given by triples $(W,\elll_W,\beta)$ where $W\subset D^{2n,\infty}$ is a
  smooth submanifold of type $W_{g,1}$ satisfying the above conditions near the
  boundary, $\elll_W$ is a $\theta$-framing of $W$ such that for some
  identification of $W$ with $W_{g,1}$, the pullback of $\elll_W$ lies in
  $\Fr^\theta_\partial(W_{g,1},\ell_{g,1})$, and $\beta\in E_{2n,\infty}(r)$
  satisfying
  $W\cap \beta(\ul{r}\times D^{2n,\infty})=\beta(\ul{r}\times D^{2n})$ and
  $\elll_W\circ T\beta|_{\ul{r}\times D^{2n}}=\ul{r}\times \ell_{0,1}$.
\end{constr}

\begin{constr}
  For any $r,r_1,\dotsc,r_r,g,g_1,\dotsc,g_r\ge 0$, we have a map
  \[\caW^\theta_g(r)\times \prod_{i=1}^r \caW^\theta_{g_i}(r_i)\to \caW^\theta_{g+\sum_i g_i}({\textstyle\sum_i r_i})\]
  as follows: Let $(W,\elll_W,\beta)\in \caW^\theta_g(r)$ and
  $(W_i,\elll_{W_i},\beta_i)\in \caW^\theta_{h_i}(r_i)$, then the result is
  given by
  $(\hat W,\elll_{\hat W},\beta\circ (\beta_1\sqcup\dotsb\sqcup \beta_r))$,
  where
  \[\hat W\coloneqq \bigl(W\setminus \beta(\ul{r}\times \tfrac12D^{2n})\bigr)\cup_{\beta(\ul{r}\times C)} \bigcup_{i=1}^r \beta(\{i\}\times W_i).\]%
  and $\elll_{\smash{\hat W}\vphantom{W}}$ is defined to be $\elll_W$ on
  $W\setminus\beta(\ul{r}\times \frac12D^{2n})$ and
  $\elll_{W_i}\circ T\beta^{-1}$ on each $\beta(\{i\}\times W_i)$, noting that
  both terms agree on the intersection of their domains. Using a diffeomorphism
  of $W$ that moves the discs $\beta(\ul{r}\times D^{2n})$ near the collar, we
  obtain an identification of $\hat W$ with $W_{g+h_1\dotsb+h_r,1}$ that pulls
  back $\elll_{\smash{\hat W}\vphantom{W}}$ to a $\theta$-framing in the $\Diff_\partial$-orbit of
  $\ell_{\smash{g+h_1+\dotsb+h_r,1}}$.

  This turns the collection of
  $\caW^\theta(r)\coloneqq \coprod_{g\ge 0}\caW^\theta_g(r)$ into a proper
  operad, called the \emph{$\theta$-framed generalised surface operad}, with
  identity
  $(D^{2n},\ell_{0,1},\on{id}_{\smash{D^{2n,\infty}}})\in \caW^\theta_0(1)$.%
\end{constr}

\begin{constr}\label{constr:f2n}
  Recall that a $\theta$-framed embedding
  $(\alpha,\gamma)\in \dE[\theta]{2n}(r)$ contains Moore paths
  $\gamma\in \Fr^\theta_\partial(\ul{r}\times D^{2n})^{[0,\infty)}$ which are,
  on each $\{i\}\times D^{2n}$ of a given length $t^i\ge 0$. Let
  $F^\theta_{2n}(r)\subset\dE[\theta]{2n}(r)\times
  \Emb_\partial(D^{2n},D^{2n,\infty})\times [0,\infty)\times
  \Fr^\theta_\partial(D^{2n})^{[0,\infty)}$ be the subspace containing all
  $(\alpha,\gamma,\eta,u,\zeta)$ that satisfy the following:
  \begin{itemize}
  \item There is a $\beta\in E_{2n,\infty}(r)$ that satisfies
    $\eta(D^{2n})\cap \beta(\ul{r}\times D^{2n,\infty})=\beta(\ul{r}\times
    D^{2n})$ and $\eta\circ\alpha=\beta|_{\ul{r}\times D^{2n}}$. Note
    that if such a $\beta$ exists, it is uniquely determined.
  \item $\zeta$ is constant at $[u,\infty)$, $\zeta(0)=\ell_{0,1}$ and
    $\zeta(s)\circ T\alpha=\gamma(t^i-s)$ for all $s\ge 0$ (where we put
    $\gamma(s)\coloneqq \gamma(0)$ for $s<0$) on $\{i\}\times D^{2n}$; in
    particular, $\zeta(u)\circ T\alpha=\ul{r}\times \ell_{0,1}$.
  \end{itemize}
  As done before, we skip the entry $u$ from the tuple and regard $\zeta$ as a path
  that is defined on $[0,u]$.  Then we have an operadic composition
  $F_{2n}^\theta(r)\times \prod_{i=1}^rF_{2n}^\theta(r_i)\to
  F_{2n}^\theta(\sum_i r_i)$ by taking $(\alpha,\gamma,\eta,\zeta)$ and
  $(\alpha_i,\gamma_i,\eta_i,\zeta_i)$ to
  $((\alpha,\gamma)\circ \bigsqcup_i (\alpha_i,\gamma_i),\hat\eta,\hat\zeta)$,
  where $\hat\eta$ is defined to be $\beta\circ\eta_i\circ\alpha^{-1}$ inside
  $\alpha(\{i\}\times D^{2n})$, and $\eta$ everywhere else, and where
  \[\hat\zeta(s)\coloneqq
    \begin{cases}
      \zeta_i(s-t^i)\circ (T\alpha)^{-1} & \text{inside $\alpha(\{i\}\times D^{2n})$ if $s\ge t^i$,}\\
      \zeta(s) &\text{else}.
    \end{cases}
  \]
  We have a map of proper operads $\dF[\theta]{2n}\to \dE[\theta]{2n}$ that
  remembers $(\alpha,\gamma)$ from each tuple.  Moreover, we have maps
  $\dF[\theta]{2n}\to \tilde\caW^\theta_0(r)$ by taking
  $(\alpha,\gamma,\eta,\zeta)$ to $(\eta,\zeta(u),\beta)$, where $\beta$ is
  uniquely determined as described above. Quotienting out by the action of
  $\Diff_\partial(W_{g,1})$, the above assigment constitutes a map of proper
  operads $\dF[\theta]{2n}\to\caW^\theta$.\looseness-1
\end{constr}

\begin{constr}
  Recall the operad map $\imath\colon E_{2n}\to \dE[\theta]{2n}$ from
  \cref{ex:unfr}.  For any $1\le k\le 2n$ we let $F_k(r)$ be the subspace of
  $E_k(r)\times \Fr^\theta_\partial(D^{2n})^{[0,\infty)}$ containing all
  $(\alpha,\zeta)$ such that if
  $\alpha'\colon \ul{r}\times D^{2n}\hookrightarrow D^{2n}$ denotes the
  canonical extension, then $(\imath(\alpha'),\eta_0,\zeta)$ lies in
  $\dF[\theta]{2n}(r)$, where $\eta_0\colon D^{2n}\hookrightarrow D^{2n,\infty}$
  is the standard inclusion. Then the above composition law for paths $\zeta$
  turns $F_k$ into a proper operad.

  We have operad inclusions $F_1\to F_2\to\dotsb\to F_{2n}$ and an operad map
  $F_{2n}\to\dF[\theta]{2n}$ by taking $(\alpha,\zeta)$ to
  $(\imath(\alpha'),\eta_0,\zeta)$.  Moreover, we have maps $F_k\to E_k$ by
  taking $(\alpha,\zeta)$ to $\alpha$.
\end{constr}

\begin{constr}\label{constr:stab}
  We fix, once and for all, a nullary operation $\mathbf{0}\in F_1(0)$, e.g.\
  the path in $\Fr^\theta_\partial(D^{2n})$ that is constantly
  $\ell_{0,1}$. Along the operad map $F_1\to \caW^\theta(r)$, this becomes a
  genus-$0$ nullary operation in $\caW^\theta(r)$ and we obtain \emph{capping
    maps} $\on{cap}\colon\caW^\theta_g(r)\to \caW^\theta_g(0)$ by precomposing
  each operation with $r$ copies of $\mathbf{0}$.

  Similarly, we by fixing a genus-$1$ unary operation $\sigma\in \caW^\theta_1(1)$,
  we obtain \emph{stabilisation maps}
  $\on{stab}\colon \caW^\theta_g(r)\to \caW^\theta_{g+1}(r)$ by postcomposing
  with $s$. Since stabilisation commutes with capping, we have a
  stable capping map
  $\smash{\on{hocolim}_{g\to\infty}(\on{cap}\colon \caW^\theta_g(r)\to
  \caW^\theta_g(0))}$.\looseness-1
\end{constr}

\subsection{Properties of the new operads}

The goal of this subsection is to compare the spaces $\caW_g^\theta(r)$ to
moduli spaces of $\theta$-framed manifolds with multiple boundary components,
and to show that the operad maps $\dF[\theta]{2n}\to \dE[\theta]{2n}$ and
$F_k\to E_k$ are equivalences.

\begin{defi}
  For $g,r\ge 0$, we fix an embedding
  $\alpha\colon \ul{r}\times D^{2n}\hookrightarrow \mathring W_{g,1}$ and
  consider the $\theta$-framed manifold $(W_{g,r+1},\ell_{g,r+1})$ arising from
  $(W_{g,1},\ell_{g,1})$ by removing $\alpha(\ul{r}\times \mathring D^{2n})$.

  We have inclusions
  $\Diff_\partial(W_{g,r+1})\hookrightarrow \Diff_\partial(W_{g,1})$ by
  extending diffeomorphisms trivially on
  $\alpha(\ul{r}\times \mathring D^{2n})$, and similarly
  $\Fr^\theta_\partial(W_{g,r+1})\hookrightarrow\Fr^\theta_\partial(W_{g,1})$;
  their images are given by the subspace of diffeomorphisms $\phi$ with
  $\phi\circ\alpha=\alpha$ and framings $\ell$ with
  $\ell\circ T\alpha=\ell_{g,1}\circ T\alpha$.
  We suggestively denote these maps by
  $\phi\mapsto \phi\cup_\partial (\ul{r}\times D^{2n})$ and
  $\ell\mapsto \ell\cup_\partial(\ul{r}\times D^{2n})$.\looseness-1
\end{defi}

\begin{lem}\label{lem:pcFr}
  Let $\theta\colon L\to \on{BO}(2n)$ be $\pi_1$-injective and let
  $\ell\in \Fr^\theta_\partial(W_{g,r+1})$.  Then  $\ell$ lies in
  $\Fr^\theta_\partial(W_{g,r+1},\ell_{g,r+1})$ if and only if
  $\ell\cup_\partial(\ul{r}\times D^{2n})$ lies in
  $\Fr^\theta_\partial(W_{g,1},\ell_{g,1})$.
\end{lem}
\begin{proof}
  If $\ell\in\Fr^\theta_\partial(W_{g,r+1},\ell_{g,r+1})$, then we find
  $\psi\in \Diff_\partial(W_{g,r+1})$ and a path
  $\gamma$ from $\ell$ to\linebreak $\ell_{g,r+1}\circ T\psi$, so
  $s\mapsto \gamma(s)\cup_\partial (\ul{r}\times D^{2n})$ is a path from
  $\ell\cup_\partial(\ul{r}\times D^{2n})$ to
  $\ell_{g,1}\circ T(\psi\cup_\partial\ul{r}\times D^{2n})$.\looseness-1

  For the converse, we can assume that $\alpha$ extends to an embedding
  $\tilde\alpha\colon\ul{r}\times 2D^n\hookrightarrow W_{g,1}$ and
  $\ell_{g,1}\circ T\tilde\alpha$ factors over the trivial $\R^{2n}$-bundle over
  a point on each $\tilde\alpha(\{i\}\times \mathring 2D^{2n})$.  Now we note
  that $\Fr^\theta_\partial(W_{g,r+1})\to \Fr^\theta_\partial(W_{g,1})$ is the
  (homotopy) fibre of $(-)\circ T\alpha$, which lands in
  $\Fr^\theta(\ul{r}\times D^{2n})=\Fr^\theta(D^{2n})^r$. As in the proof of
  \cref{lem:unaries}, this space fits into a fibre sequence
  $\on{O}(2n)^r\to \Fr^\theta(D^{2n})^r\to L^r$, which can be continued by
  $\theta^r\colon L^r\to \B\on{O}(2n)^r$. Since $\theta$ is assumed to be
  $\pi_1$-injective, the map $\pi_1\on{O}(2n)^r\to \pi_1\Fr^\theta(D^{2n})^r$ is
  surjective.

  Postcomposing with the connecting morphism
  $\pi_1\Fr^\theta(D^{2n})^r\to \pi_0\Fr^\theta_\partial(W_{g,r+1})$ of the
  first fibre sequence, we obtain a map
  $\Phi\colon \pi_1\on{O}(2n)^r\to \pi_0\Fr^\theta_\partial(W_{g,r+1})$ and by
  exactness, this shows that if a component of $\Fr^\theta_\partial(W_{g,r+1})$
  gets identified with the component of $\ell_{g,r+1}$ when applying
  $({-})\cup_\partial(\ul{r}\times D^{2n})$, then it lies in the image of
  $\Phi$. This map $\Phi$, in turn, sends each loop
  $\gamma^\bullet\colon [0,1]\to \on{O}(2n)^r$ to the component of
  $\ell^\gamma$, which is given as follows: Let
  $\smash{\hat\alpha\colon \ul{r}\times S^{2n-1}\times [0,1]\cong \ul{r}\times
    (2D^{2n}\setminus \mathring D^{2n})\hookrightarrow W_{g,r+1}}$ be the
  restriction of $\tilde\alpha$, then a trivialisation of
  $\tilde\alpha^*TW_{g,1}$ restricts to
  $\hat\alpha^*TW_{g,r+1}\cong(\ul{r}\times S^{2n-1}\times [0,1])\times\R^{2n}$.
  Under these identifications, $\ell^\gamma$ is given by perturbing
  $\ell_{g,r+1}$ by the vector bundle automorphism
  $(i,z,t,v)\mapsto (i,z,t,\gamma^i(t)\cdot v)$ inside
  $\hat\alpha(\ul{r}\times S^{2n-1}\times [0,1])$.
  Now we note that the morphism $\pi_1\on{SO}(2)^r\to \pi_1\on{O}(2n)^r$ is surjective, so we
  may assume that each $\gamma^i$ is a mere rotation. Then we
  find a diffeomorphism $\psi\in\Diff_\partial(W_{g,r+1})$, comprised of $r$
  Dehn twists near the $r$ boundary spheres, such that
  $\ell_{g,r+1}^\gamma$ lies in the same component as $\ell_{g,r+1}\circ T\psi$.\looseness-1

  Coming back to the statement, let $\phi\in\Diff_\partial(W_{g,1})$ such that
  $\ell\cup_\partial(\ul{r}\times D^{2n})$ lies in the same component as
  $\ell_{g,1}\circ T\phi$. Then $\phi$ can be isotoped to a diffeomorphism
  preserving $\alpha$, and since this new diffeomorphism still satisfies the
  above condition, we can assume that $\phi$ itself is of the form
  $\psi\cup_\partial(\ul{r}\times D^{2n})$ for some
  $\psi\in\Diff_\partial(W_{g,r+1})$. Then the framing
  $(\ell\circ T\psi^{-1})\cup_\partial (\ul{r}\times D^{2n})$ lies in
  the same component as $\ell_{g,1}$. By the previous discussion, we find a
  diffeomorphism $\psi'\in\Diff_\partial(W_{g,r+1})$ such that
  $\ell\circ T(\psi^{-1}\circ\psi')$ lies in
  the same component as $\ell_{g,r+1}$.\looseness-1
\end{proof}

\begin{lem}\label{lem:moduli}
  If $L$ is $\pi_1$-injective, then we have equivalences
  $\caW^\theta_g(r)\simeq \caM^\theta_\partial(W_{g,r+1},\ell_{g,r+1})$.
\end{lem}
\begin{proof}
  Let $E'\subset E_{2n,\infty}(r)$ be the subspace of all $\beta$ with
  $\beta(\ul{r}\times 2D^{2n}\times D^\infty)\subseteq \frac12 D^{2n}\times
  D^\infty$, where we extend $\beta$ on $\ul{r}\times 2D^{2n}\times D^\infty$
  with the same term; and we let $\caW'\subset \caW^\theta_g(r)$ be the subspace
  of tuples $(W,\elll_W,\beta)$ with $\beta\in E'$. Then $\caW^\theta_g(r)$
  deformation retracts onto $\caW'$ by shrinking the radii of the given
  embeddings and rescaling $\elll_W$ close to the discs.  One can now easily
  construct, around any $\beta\in E'$, an open neighbourhood $U\subset E'$ and a
  map $\tau\colon U\to \Diff(D^{2n,\infty})$ such that, for each
  $\hat\beta\in U$, the diffeomorphism $\tau(\hat\beta)$ restricts to the
  identity on $C\times D^\infty$ and satisfies
  $\tau(\hat\beta)\circ\hat\beta=\beta$. These maps can then be employed to show
  that $\caW'\to E'$ is locally trivial and hence a fibration. As the target
  space $E'$ is still contractible, it hence suffices to study the actual fibre
  of a fixed element $\beta\in E'$.\looseness-1

  This fibre is given by the (strict) $\Diff_\partial(W_{g,1})$-quotient of the
  subspace $S$ of the product\linebreak
  $\Emb_\partial(W_{g,1},D^{2n,\infty})\times
  \Fr^\theta_\partial(W_{g,1},\ell_{g,1})$ of all $(\eta,\elll)$ with
  $\eta(W_{g,1})\cap \beta(\ul{r}\times D^{2n,\infty})=\beta(\ul{r}\times
  D^{2n})$ and
  $\elll\circ T\eta^{-1}\circ T\beta|_{\ul{r}\times D^{2n}}=\ul{r}\times
  \ell_{0,1}$.\looseness-1

  We fix an embedding
  $\alpha\colon \ul{r}\times D^{2n}\hookrightarrow \mathring W_{g,1}$. We note
  that the subgroup of $\Diff_\partial(W_{g,1})$ containing all $\phi$ that
  preserve $\alpha$ agrees with the diffeomorphism group
  $\Diff_\partial(W_{g,r+1})$, and it acts on the subspace $S'\subseteq S$ of
  all $(\eta,\elll)$ where the first of the above conditions is replaced by the
  stronger assumption $\eta\circ \alpha=\beta|_{\ul{r}\times D^{2n}}$; then the
  second assumption reads $\elll\circ T\alpha=\ul{r}\times \ell_{0,1}$.  We
  moreover note that the map $S'\times\Diff_\partial(W_{g,1})\to S$ taking
  $(\eta,\elll,\phi)$ to $(\eta\circ \phi,\elll\circ T\phi)$ is a
  $\Diff_\partial(W_{g,1})$-equivariant principal
  $\Diff_\partial(W_{g,r+1})$-bundle (where $\Diff_\partial(W_{g,r+1})$ acts by
  $\psi\cdot (\eta,\elll,\phi)=(\eta\circ\psi,\elll\circ T\psi,\psi^{-1}\circ
  \phi)$ and where $\Diff_\partial(W_{g,1})$ right acts on $\Diff_\partial(W_{g,1})$
  and $S$ by precomposition), so it induces a homeomorphism
  \[S'/\Diff_\partial(W_{g,r+1})=(S'\times_{\Diff_\partial(W_{g,r+1})}%
    \Diff_\partial(W_{g,1}))/\Diff_\partial(W_{g,1})\to
    S/\Diff_\partial(W_{g,1}).\]%
  Finally, we note that $S'$ is the product two spaces: the space of embeddings
  of $W_{g,r+1}$ into $D^{2n,\infty}$ with a fixed boundary behaviour, and the
  subspace of $\Fr^\theta_\partial(W_{g,r+1})$ of all $\theta$-framings $\ell$
  such that
  $\ell\cup_\partial(\ul{r}\times D^{2n})\in
  \Fr^\theta_\partial(W_{g,1},\ell_{g,1})$. The first factor is contractible by
  an application of Whitney’s embedding theorem and the second factor is
  identified with $\Fr^\theta_\partial(W_{g,r+1},\ell_{g,r+1})$ by
  \cref{lem:pcFr}.  As the action of $\Diff_\partial(W_{g,r+1})$ on the first
  factor is free and proper, we hence get the desired equivalence
  \begin{align*}
    S'/\Diff_\partial(W_{g,r+1}) &\simeq S'\sslash \Diff_\partial(W_{g,r+1})
    \\ &\simeq \Fr^\theta_\partial(W_{g,r+1},\ell_{g,r+1})\sslash \Diff_\partial(W_{g,r+1})
    \\ &=\caM^\theta_\partial(W_{g,r+1},\ell_{g,r+1}).\qedhere
  \end{align*}
\end{proof}

\begin{cor}\label{cor:Wgood}
  Let $\theta\colon L\to \B\on{O}(2n)$ be a spherical tangential
  structure with $\pi_1$-injective $\theta$.
  \begin{enumerate}
  \item If $2n\ge 6$ (or $2n=2$ and $\theta$ is admissible to
    \cite[Thm.\,7.1]{RW-2016}), then the stabilisation maps
    $\smash{\on{stab}\colon\caW^\theta_g(r)\to \caW^\theta_{g+1}(r)}$ induce
    isomorphisms in $H_i(-;\Z)$ for $\smash{i\le \frac12\kern1px g-\frac32}$
    (for $2n=2$, the slope has to be replaced by the one from
    \cite[Thm.\,7.1]{RW-2016}) for any $r$.
  \item If $L$ is $n$-connected, then the stable capping map
    $\on{hocolim}_{g\to\infty}(\on{cap}\colon\caW^\theta_g(r)\to
    \caW^\theta_g(0))$ is a homology equivalence for each $r\ge 0$.
  \end{enumerate}
\end{cor}
\begin{proof}
  We note that under the equivalence of \cref{lem:moduli}, the stabilisation and
  capping maps are identified up to homotopy with the usual maps appearing in
  the classical stability theorems. Then statement 1 is Harer’s stability
  theorem \cite{Harer-1984} in the version of \cite[Thm.\,7.1]{RW-2016} for
  $2n=2$ and \cite[Thm.\,1.4]{Galatius-RW-2018} for $2n\ge 6$, using that
  $\theta$ is spherical to get the improved slope and that
  $\ell_{1,1}$ is admissible in the sense of \cite[Def.\,1.3]{Galatius-RW-2018}.

  Statement 2 is implied by \cite[Thm.\,1.3]{Galatius-RW-2017}, noting that
  $\caM^\theta_\partial(W_{g,r+1},\ell_{g,r+1})$ is a single path component of
  what they would call $\caN^\theta_n(\ul{\smash{r+1}}\times S^{2n-1},\ul{\smash{r+1}}
  \times\ell_{0,1}|_{S^{2n-1}})$. Here
  $n$-connectivity of $L$ is needed as the bundle maps
  $\ell\colon TW_{g,r+1}\to \theta^*V_{2n}$ constituting
  $\caN^\theta_n(W_{g,r+1},\ell_{g,r+1})$ are required to cover an $n$-connected
  map $W_{g,r+1}\to L$ among base spaces. Since $W_{g,r+1}$ is
  $(n-1)$-connected, this condition is automatically satisfied if $L$ is
  $n$-connected, see also \cite[§\,7.2]{BBPTY} for a similar argument.
\end{proof}

\begin{lem}\label{lem:eqF2n}
  The operad maps $\dF[\theta]{2n}\to \dE[\theta]{2n}$ and $F_k\to E_k$ are
  equivalences of proper operads.
\end{lem}
\begin{proof}
  We shall only prove that the first map is an equivalence, as the argument for
  $F_k\to E_k$ is nearly identical. If $E'\subset \dE[\theta]{2n}(r)$ contains all
  $(\alpha,\gamma)$ with
  $\alpha(\ul{r}\times \frac12D^{2n})\subseteq \frac12D^{2n}$, then $E'$ is a
  deformation retract of $\dE[\theta]{2n}(r)$ and the map
  $\dF[\theta]{2n}(r)\to \dE[\theta]{2n}(r)$ corestricts to a fibration over
  $E'$. For any $(\alpha,\gamma)\in E'$, the actual fibre
  $\dF[\theta]{2n}(r)_{\alpha,\gamma}$ is given by the space of tuples
  $(\alpha,\gamma,\eta,\zeta)\in \dF[\theta]{2n}$ satisfying the two conditions
  from \cref{constr:f2n}, and it is our aim to show that this space is
  contractible.

  To do so, we note that the map
  $\dF[\theta]{2n}(r)_{\alpha,\gamma}\to E_{2n,\infty}(r)$ assigning to any such
  tuple the (unique) $\beta$ with $\eta\circ\alpha=\beta|_{\ul{r}\times D^{2n}}$
  is again locally trivial, and therefore a fibration. Since $E_{2n,\infty}(r)$
  is contractible, it hence suffices to study its fibre. Similar as in the
  previous proof, we see that it is the product of two spaces:
  \begin{itemize}[itemsep=0em]
  \item The space of embeddings of $W_{g,r+1}$ into $D^{2n,\infty}$ with a fixed
    boundary behaviour. This factor is contractible by Whitney’s embedding
    theorem.
  \item The space $S$ of Moore paths
    $\zeta\colon [0,u]\to \Fr^\theta_\partial(D^{2n})$ for some $u\ge 0$,
    satisfying $\zeta(0)=\ell_{0,1}$ and $\zeta(s)\circ T\alpha=\gamma(t_i-s)$.
  \end{itemize}
  It hence suffices to show that the second factor $S$ is contractible. Being
  only interested in its homotopy type, we may replace Moore paths by paths of
  length $1$. If $P_xX$ denotes the space of paths starting at $x\in X$ for any
  space $X$, then we see that $S$ is the fibre of
  \[({-})\circ T\alpha\colon P_{\ell_{0,1}}\Fr^\theta_\partial(D^{2n})\to
    P_{\ell_{0,1}\circ T\alpha}\Fr^\theta_\partial(\ul{r}\times D^{2n}),\]%
  the fibre taken at the reversed paths $\gamma$. We conclude the argument by
  noting that both source and target of this map are contractible, and since
  $\alpha$ is a cofibration, the restriction $({-})\circ T\alpha$ is a
  fibration, so the fibre $S$ coincides with the contractible homotopy fibre.
\end{proof}

\begin{cor}$F_1$ is an $A_\infty$-operad and the operad map
  $F_1\to \dF[\theta]{2n}$ sends both path components of $F_1(2)$ to the same
  component. Moreover, $\dF[\theta]{2n}(0)\simeq \dE[\theta]{2n}(0)=*$ is
  contractible.
\end{cor}

Through the operad map $F_1\to \caW^\theta$, each $\caW^\theta$-algebra $A$ is
in particular an $A_\infty$-algebra. In the case of $A=\caW^\theta(0)$, we get
the following:

\begin{cor}\label{cor:MT}
  If $L$ is $n$-connected, then the $A_\infty$-algebra $\caW^\theta(0)$
  group-completes to the infinite loop space $\Z\times\Omega^\infty_0\MT\theta$, where
  $\MT\theta$ is the tangential Thom spectrum.
\end{cor}
\begin{proof}
  Recall that \cref{lem:moduli} provides an equivalence
  $\caW^\theta(0)\simeq\coprod_{g\ge 0}\caM^\theta_\partial(W_{g,1},\ell_{g,1})$
  (this does not even use that $\theta$ is $\pi_1$-injective – which is
  satisfied anyway, as $L$ is $n$-connected).  By the group-completion theorem,
  it follows that $\Omega\B\caW^\theta(0)$ is equivalent to
  $\Z\times\on{hocolim}_{g\to\infty}\caM^\theta_\partial(W_{g,1},\ell_{g,1})^+$.
  Again, using that any map $W_{g,1}\to L$ is $n$-connected, $\caW^\theta(0)$ is
  a union of several path-components of
  $\caN^\theta_n(S^{2n-1},\ell_{0,1}|_{S^{2n-1}})$ from the setting of
  \cite{Galatius-RW-2017}, and the stabilisations from \cite{Galatius-RW-2017}
  restrict on path-components to stabilisations among
  $\caM^\theta_\partial(W_{g,1},\ell_{g,1})$. We therefore can apply
  \cite[Thm.\,1.5]{Galatius-RW-2017} and conclude that the Pontrjagin–Thom map
  $\on{hocolim}_{g\to\infty}\caM^\theta_\partial(W_{g,1},\ell_{g,1})\to
  \Omega_0^\infty\MT\theta$ is acyclic. By applying the Quillen
  plus-construction to the left side, we get an equivalence.\looseness-1
\end{proof}

In the case of $2n=2$ and
$\theta\colon \on{BSO}(2)\to\on{BO}(2)$, \cref{cor:MT} is an instance of the
Madsen–Weiss theorem \cite{Madsen-Weiss} telling us
$\Omega\B\coprod_{g\ge
  0}\B\Diff^+_\partial(S_{g,1})\simeq\Omega^\infty\MT\on{SO}(2)$.

\subsection{Decorated moduli spaces as a pushforward}

\begin{rmk}
  As the operad map $\dF[\theta]{2n}\to\caW^\theta$ lands in the genus-$0$
  components of $\caW^\theta$, the right $\dbF[\theta]{2n}$-functor $\bW^\theta$
  splits as a disjoint union $\coprod_{g\ge 0}\dbW[\theta]g$. Thus, the
  pushforward $\caW^\theta\otimes^\bbL_{\dF[\theta]{2n}}({-})$ splits as a
  disjoint union of bar constructions
  $|\B_\bullet(\dbW[\theta]g,\dbF[\theta]{2n},{-})|$. This is the ‘grading by
  genus’ that we promised at the beginning of this section.
\end{rmk}

The aim of this subsection is to give the proof of \cref{prop:A}, i.e.\ to provide, for each
$\dE[\theta]{2n}$-algebra $A$ and each $g\ge 0$, equivalences
$W_{g,1}^\theta[A]\simeq
|\B_\bullet(\dbW[\theta]g,\dbF[\theta]{2n},\kappa^*A)|$.

\begin{constr}\label{constr:rmod}
  Recall the spaces $\tilde\caW^\theta_g(r)$ from \cref{constr:tW}. For any
  $g\ge 0$, the family $(\tilde\caW^\theta_g(r))_{r\ge 0}$ is a right
  $\dF[\theta]{2n}$-module: We have maps
  $\tilde\caW^\theta_g(r)\times \prod_{i}\dF[\theta]{2n}(r_i)\to
  \tilde\caW^\theta_g(\sum_i r_i)$ taking $(\eta,\elll,\beta)$ and
  $(\alpha_i,\gamma_i,\eta_i,\zeta_i)$ to
  $(\hat\eta,\hat\elll,\beta\circ (\beta_1\sqcup\dotsb\sqcup \beta_r))$, where
  $\beta_i$ are uniquely determined as before, and where $\hat\eta$ is
  $\beta\circ \eta_i\circ\beta^{-1}\circ \eta$ inside
  $\eta^{-1}(\beta(\{i\}\times D^{2n})$ and $\eta$ elsewhere, and $\hat\elll$ is
  $\zeta_i(u_i)\circ T\beta|_{\smash{\ul{r}\times D^{2n}}}^{-1}\circ T\eta$
  inside each $\eta^{-1}(\beta(\{i\}\times D^{2n})$ and $\elll$ else.

  These structure maps are equivariant with respect to the action of
  $\Diff_\partial(W_{g,1})$ on each $\tilde\caW^\theta_g(r)$, and the induced
  map
  $\caW^\theta_g(r)\times\prod_i \dF[\theta]{2n}(r_i)\to \caW^\theta_g(\sum_i
  r_i)$ is part of the operadic composition of $\caW^\theta$, precomposed with
  the operad map $\dF[\theta]{2n}\to\caW^\theta$.\looseness-1
\end{constr}

\begin{proof}[Proof of \cref{prop:A}]
  As $\dF[\theta]{2n}\to \dE[\theta]{2n}$ is an equivalence of proper
  operads, the induced map on bar constructions
  $|\B_\bullet(\mdE[\theta]{W_{g,1}},\dbF[\theta]{2n},\kappa^*A)|\to
  |\B_\bullet(\mdE[\theta]{W_{g,1}},\dbE[\theta]{2n},A)|$ is an equivalence.  It
  is hence our goal to establish an equivalence of
  $\dbF[\theta]{2n}$-functors between $\dbW[\theta]g$ and
  $\mdE[\theta]{W_{g,1}}$.

  To this end, we note that the $\Diff_\partial(W_{g,1})$-equivariant right
  $\dF[\theta]{2n}$-module structure on $(\tilde\caW^\theta_g(r))_{r\ge 0}$
  gives rise to a right $\dbF[\theta]{2n}$-functor $\tilde\bW^\theta$ with
  values in $\Diff_\partial(W_{g,1})$-spaces, and since the action of
  $\Diff_\partial(W_{g,1})$ on each $\tilde\caW^\theta(r)$ is free and proper,
  the canonical augmentation of $\dbF[\theta]{2n}$-functors
  $\tilde\bW^\theta_g\sslash \Diff_\partial(W_{g,1})\to
  \tilde\bW^\theta_g/\Diff_\partial(W_{g,1}) = \bW^\theta_g$ is an equivalence.

  Next, we generalise \cref{constr:f2n} in two ways: We consider embeddings into
  $W_{g,1}$ instead of only into $D^{2n}=W_{0,1}$, and we let the framing of
  $W_{g,1}$ vary. More precisely, let
  $\caV_g(r)\subset \ul\Emb^\theta(\ul{r}\times D^{2n},W_{g,1})\times
  \Emb_\partial(W_{g,1},D^{2n,\infty})\times [0,\infty)\times
  \Fr^\theta_\partial(W_{g,1},\ell_{g,1})^{[0,\infty)}$ contain all tuples
  $(\alpha,\gamma,\ell,\eta,u,\zeta)$ satisfying the very same conditions as in \cref{constr:f2n}
  we only replace $\zeta(0)=\ell_{0,1}$ by $\zeta(0)=\ell$.
  We have a map $\caV_g(r)\to \tilde\caW^\theta_g(r)$ taking such a tuple
  to $(\eta,\zeta(u),\beta)$, and we have a map
  $\smash{\caV_g(r)\to \ul\Emb^\theta(\ul{r}\times D^{2n},W_{g,1})}$ that only remembers
  $(\alpha,\gamma,\ell)$. Both maps are equivalences: For the second map, we can
  run the same proof as \cref{lem:eqF2n}, and for the first one, we note that
  the fibre of any $(\eta,\elll,\beta)$ is parametrised by the contractible
  space of real numbers $t^1,\dotsc,t^r\ge 0$ and paths $\zeta$ inside
  $\Fr^\theta_\partial(W_{g,1},\ell_{g,1})$ \emph{ending} in $\elll$, such that
  $\zeta(s)\circ T\eta^{-1}\circ T\beta|_{\ul{r}\times D^{2n}}=\ell_{0,1}$ holds
  inside $\{i\}\times D^{2n}$ for all $s\ge t_i$.

  Both maps are morphisms of right $\dF[\theta]{2n}$-modules, and hence induce
  morphisms of $\dbF[\theta]{2n}$-functors
  $\udbE[\theta]{W_{g,1}}\leftarrow \bV_g\to \tilde\bW^\theta_g$. As
  all three $\dbF[\theta]{2n}$-functors attain values in
  $\Diff_\partial(W_{g,1})$-spaces and the maps are equivariant, we obtain the
  desired zig-zag of $\dF[\theta]{2n}$-functors
  \[\mdE[\theta]{W_{g,1}}\longleftarrow \bV_g\sslash \Diff_\partial(W_{g,1})\longrightarrow
    \tilde\bW_g^\theta\sslash \Diff_\partial(W_{g,1})\longrightarrow \bW_g^\theta.\] By
  applying $\B_\bullet({-},\dbF[\theta]{2n},\kappa^* A)$, we get a
  zig-zag of levelwise equivalences of simplicial spaces, all of which are
  proper as all components of the above right $\dF[\theta]{2n}$-modules are
  Hausdorff and $\frS$-free. This gives the desired equivalence on geometric
  realisations.
\end{proof}

\begin{constr}\label{constr:Wgstab}
  Using the equivalence from \cref{prop:A}, the unary operation
  $\sigma\in \caW^\theta_1(1)$ induces stabilisation maps
  $\smash{W^\theta_{g,1}[A]\to W^\theta_{g+1,1}[A]}$. The equivalence from
  \cref{prop:A} also tells us that the homotopy type of $\dW[\theta]{*,1}[A]$
  carries an $A_\infty$-structure, and by abuse of notation, we may write
  $\Omega\B\dW[\theta]{*,1}\coloneqq
  \Omega\B(\caW^\theta\otimes^\bbL_{\dF[\theta]{2n}}\kappa^*A)$.\looseness-1

  If $A$ is path-connected, then $W^\theta_{g,1}[A]$ is path-connected by
  \cref{lem:Actd}, so by the group-completion theorem, we obtain an equivalence
  \[\textstyle
    \on{hocolim}_{g\to\infty}W_{g,1}^\theta[A]^+\simeq\Omega_0\B(\caW^\theta\otimes^\bbL_{\dF[\theta]{2n}}\kappa^*
    A)=:\Omega_0\B\dW[\theta]{*,1}[A].\]
\end{constr}

We close this section by noting that we can use \cref{prop:A} to deduce
homological stability of $\dW[\theta]{\bullet,1}[A]$, generalising
\cite[Thm.\,\textsc{e}]{Bonatto}, from the stability results for moduli spaces
of surfaces \cite{Harer-1984,RW-2016} and of high-dimensional manifolds
\cite{Galatius-RW-2018}:

\begin{thm}\label{thm:A}
  Let $\theta\colon L\to \B\on{O}(2n)$ be a spherical tangential structure with
  $\pi_1$-injective $\theta$, and let $A$ be an $\dE[\theta]{2n}$-algebra.  If
  $2n\ge 6$ (or $2n=2$ and $\theta$ is admissible to \cite[Thm.\,7.1]{RW-2016}),
  then the maps $W^\theta_{g,1}[A]\to W^\theta_{g+1,1}[A]$ induce isomorphisms
  in $H_i(-;\Z)$ for $i\le \frac12\kern1px g-\frac32$ (for $2n=2$, the slope is
  different, but coincides with the one from \cite[Thm.\,7.1]{RW-2016}).
\end{thm}
\begin{proof}
  Let us write $i_g\coloneqq \frac12\kern1px g-\frac32$ (or the slope from
  \cite[Thm.\,7.1]{RW-2016} for $2n=2$).  The map in question is equivalent to
  the map
  $|\B_\bullet(\dbW[\theta]g,\dbF[\theta]{2n},\kappa^*A)|\to
  |\B_\bullet(\dbW[\theta]{g+1},\dbF[\theta]{2n},\kappa^*A)|$, induced by the
  morphism $\dbW[\theta]g\to \dbW[\theta]{g+1}$ of right
  $\dbE[\theta]{2n}$-functors that, on each space $X$, is a union of maps
  $\on{stab}\times_{\frS_r} \on{id}\colon \caW^\theta_g(r)\times_{\frS_r} X^r\to
  \caW^\theta_{g+1}(r)\times_{\frS_r} X^r$. Since $\on{stab}$ induces isomorphisms
  in homology in degree at most $i_g$, and $\frS_r$ acts freely on both sides,
  and $\caW^\theta_g(r)$ and $\caW^\theta_{g+1}(r)$ are Hausdorff, a spectral
  sequence argument shows that the same holds for each of the maps
  $\on{stab}\times_{\frS_r}\on{id}$.

  This shows that the map
  $\B_\bullet(\dbW[\theta]g,\dbF[\theta]{2n},\kappa^*A)\to
  \B_\bullet(\dbW[\theta]{g+1},\dbF[\theta]{2n},\kappa^*A)$ induces levelwise
  isomorphisms in homology in degrees at most $i_g$. Since both source and
  target are proper simplicial spaces, \cite[Prop.\,A1:iv]{Segal-1974} enables
  us to consider the map among ‘thick’ geometric realisations instead.  Here we
  can invoke the spectral sequence associated with the skeletal filtration, and
  a comparison argument for spectral sequences shows that the geometric
  realisation is a homology equivalence in that range as well.
\end{proof}

\section{Splitting pushforwards to operads with homological stability}
\label{sec:split}
\Cref{prop:A} motivates us to study the group-completion of
$\caW^\theta\otimes^\bbL_{\dF[\theta]{2n}} \kappa^*A$ for a spherical tangential
structure $\theta$ and an $\dE[\theta]{2n}$-algebra $A$.  Here we use that if
$\theta\colon L\to \on{BO}(2n)$ is $n$-connected, then
$\caW^\theta$ is an operad with homological stability as in
\cite{BBPTY}, and we prove a splitting result similar to \cite[§\,5]{BBPTY} and
\cite[§\,5]{Bianchi-Kranhold-Reinhold} for pushforwards.  Before
carrying out this approach, we first have to switch to a based framework.\looseness-1

\begin{constr}
  The operad $E_0$ has exactly two operations, namely the identity $\mathbf{1}$
  and a single nullary operation $\mathbf{0}$. Thus, $E_0$-algebras are the same
  as based spaces.  An operad under $E_0$ is the same as an operad $\caO$
  together with a preferred nullary operation $\mathbf{0}_\caO$.  For any operad
  $\caO$ under $E_0$, we have \emph{capping maps}
  $\on{cap}\colon\caO(r)\to \caO(0)$ given by precomposing an $r$-ary operation
  with $r$ copies of $\mathbf{0}_\caO$.\looseness-1

  Let $\caO$ be an operad under $E_0$. The forgetful functor from
  $\caO$-algebras to based spaces has a strict left-adjoint, namely
  $\tF^\caO\coloneqq \caO\otimes_{E_0}({-})$. We denote the monad associated to
  this adjunction by $\tilde \bbO$; it should be seen as a \emph{reduced}
  version of $\bbO$. More explicitly, $\tilde\bbO X$ is the (strict) coequaliser
  of the two canonical maps $\bbO \bE_0 X\rightrightarrows \bbO X$.
\end{constr}

Under mild point-set topological assumptions, we can also use these reduced
monads to describe the derived pushforward, similar to \cref{rem:freeRes}:

\begin{rmk}\label{rmk:wellP}
  If $\caP$ is proper and $A$ is a $\caP$-algebra whose underlying based space
  is well-pointed (in short, a ‘well-pointed $\caP$-algebra’), then the
  augmented simplicial $\caP$-algebra $\B_\bullet(\tF^\caP,\tilde\bbP,A)\to A$
  is a proper simplicial space (see e.g.\ \cite[Prop.\,1.4.42]{Zeman} for
  details), has an extra degeneracy, and hence is a $\tilde\bbP$-free simplicial
  resolution of $A$ in the sense of
  \cite[Def.\,8.18]{Galatius-Kupers-RW-2018}. If $\caO$ is $\frS$-free as well
  and $\caP\to \caO$ is a map of operads, then, since $\tilde\bbO$ commutes with
  geometric realisations, the left-derived pushforward
  $\caO\otimes^\bbL_\caP A$ (whose preferred point-set model for us remains
  $|\B_\bullet(\bO,\bP,A)|$ as in \cref{rem:freeRes}) is equivalent, as an
  $\caO$-algebra, to the ‘reduced’ two-sided bar construction
  $|\B_\bullet(\tilde\bbO,\tilde\bbP,A)|$.\looseness-1
\end{rmk}

\begin{lem}\label{lem:freeWell}
  Let $\caP$ a proper operad, let $\caO$ be a $\frS$-free operad under $E_0$
  such that $\{\mathbf{0}_\caO\}\hookrightarrow \caO(0)$ is a cofibration, and
  let $\caP\to\caO$ be an operad map. Moreover, let $A$ be a
  $\caP$-algebra. Then the $\caO$-algebra $\caO\otimes^\bL_\caP A$ is
  well-pointed.
\end{lem}
\begin{proof} 
  If $*_\bullet$ denotes the simplicial singleton, then the basepoint inclusion
  of $\caO\otimes^\bL_\caP A$ is the geometric realisation of the simplicial
  map $f_\bullet\colon *_\bullet\to\B_\bullet(\bbO,\bbP,A)$ that freely extends
  \[\{\mathbf{0}_\caO\}\hookrightarrow \caO(0)\subset \coprod_{r\ge
      0}\caO(r)\times_{\frS_r} A^r = \bbP A=\B_0(\bbO,\bbP,A).\]%
  As this map and all degeneracies of $\B_\bullet(\bbO,\bbP,A)$ are
  cofibrations, $f_\bullet$ is a Reedy cofibration. By
  \cite[Thm.\,18.6.7:1]{Hirschhorn}, $|f_\bullet|$ is a cofibration as well.
\end{proof}

The following definition is a special case of \cite[Def.\,4.5]{BBPTY}:

\begin{defi}
  An \emph{operad with homological stability} (\acr{OHS}) is given by:
  \begin{enumerate}
  \item a proper operad $\caO$ with a grading
    $\caO(r)=\coprod_{g\ge 0}\caO_g(r)$ such that the $\frS_r$-actions restrict
    to $\caO_g(r)$ and the composition is graded,\looseness-1
  \item an $A_\infty$-operad $\caA$, a preferred nullary
    $\mathbf{0}_\caA\in \caA(1)$, and a map of operads $\mu\colon \caA\to \caO$
    (turning $\caO$ into an operad under $E_0$) landing inside $\caO_0$, such
    that $\mu(\caA(2))\subseteq\caO_0(2)$ lies in a single path-component,
  \end{enumerate}
  and we require the following stability condition: There is a
  $\sigma\in \caO_1(1)$, inducing maps
  $\on{stab}\colon \caO_g(r)\to \caO_{g+1}(r)$ by postcomposition, such that for
  each $r$, the map of towers
  \[
    \begin{tikzcd}
      \caO_0(r)\ar{d}[swap]{\mathrm{cap}}\ar{r}{\mathrm{stab}}\ar[d] &\caO_1(r)\ar{r}{\mathrm{stab}}\ar{d}[swap]{\mathrm{cap}} & \caO_2(r)\ar[r]\ar{d}[swap]{\mathrm{cap}} &\dotsb\\
      \caO_0(0)\ar{r}[swap]{\mathrm{stab}} & \caO_1(0)\ar{r}[swap]{\mathrm{stab}} & \caO_2(0)\ar[r]& \dotsb
    \end{tikzcd}
  \]
  induces a homology equivalence on homotopy colimits. Then
  $\mu(\caA)\subseteq\caO_0$, and each $\caO$-algebra $A$ is in particular a
  homotopy-commutative $\caA$-algebra.
\end{defi}

\begin{expl}
  For each spherical tangential structure $\theta\colon L\to \on{BO}(2n)$, the
  $\theta$-framed generalised surface operad $\caW^\theta$ has a natural grading
  $\caW^\theta(r)=\coprod_{g\ge 0}\dsW[\theta]g(r)$, and it receives a map
  $\mu\colon F_1\to \dF[\theta]{2n}\to \caW^\theta$ from an
  $A_\infty$-operad. Since $2n\ge 2$, this map factors through $F_2$, and hence
  $\mu(F_1(2))\subseteq \dsW[\theta]0(2)$ lies in a single path-component.

  If $L$ is $n$-connected, then \cref{cor:Wgood}:2 tells us that
  the stability condition is satisfied as well, using the operation 
  $\sigma\in \dsW[\theta]1(1)$ from \cref{constr:stab} as a propagator. Therefore,
  $\caW^\theta$ is an operad with homological stability, compare
  \cite[Thm.\,7.4]{BBPTY}. Its initial algebra
  $\caW^\theta(0)$
  group-completes to $\Z\times\Omega^\infty_0\MT\theta$ by \cref{cor:MT}.
\end{expl}

\begin{rmk}
  Let $\caO$ be an \acr{OHS} and let $\caP$ be another operad under $E_0$
  with $\caP(0)\simeq *$, which comes with a map $\rho\colon \caP\to \caO$ of
  operads under $E_0$. Then we then automatically have
  $\rho(\caP(r))\subseteq \caO_0(r)$: Because $\caP(0)$ is connected and $\rho$
  is a map under $E_0$, implying
  $\rho(\mathbf{0}_\caP)=\mu(\mathbf{0}_\caA)\in \caO_0(0)$, we have
  $\rho(\caP(0))\subseteq \caO_0(0)$, and therefore
  \begin{align*}
    \on{cap}(\rho(\caP(r))\cap \caO_g(r))&\subseteq \rho(\on{cap}(\caP(r)))\cap \on{cap}(\caO_g(r))\\
                                         &\subseteq \rho(\caP(0))\cap \caO_g(0)\\
                                         &\subseteq \caO_0(0)\cap \caO_g(0).
  \end{align*}
\end{rmk}

\begin{constr}
  For what follows, we assume that $(\caO,\mu\colon \caA\to\caO)$ is an
  \acr{OHS} and $\caP$ is another proper operad with $\caP(0)\simeq *$ and a map
  $\rho\colon \caP\to\caO$ under $E_0$.

  We furthermore assume the existence of an operad map $\caO\to E_\infty$ under
  $E_0$.
  One way to achieve this map is to replace the above structure maps $\mu$ and
  $\rho$ by the equivalent maps $\mu\times\on{id}_{E_\infty}$ and
  $\rho\times\on{id}_{E_\infty}$, respectively, with $E_0$ being diagonally
  included into these product operads, and then consider the projection
  $\caO\times E_\infty\to E_\infty$, compare \cite[Cor.\,4.10]{BBPTY} and
  \cite[Lem.\,5.11]{Bianchi-Kranhold-Reinhold}.  As in their proof, one checks
  that this replacement neither changes a condition nor a consequence of
  \cref{prop:B}.
\end{constr}

\begin{prop}\label{prop:B}
  Let $\caO$ be an operad with homological stability, let $\caP$ be a proper
  operad with $\caP(0)\simeq *$, and let $\caP\to \caO$ be a map of operads
  under $E_0$. Then the map of $\caO$-algebras
  \[(\caO\otimes^\bbL_\caP A)\to (\caO\otimes^\bbL_\caP *) \times (\dE\infty
    \otimes^\bbL_\caP A)\]%
  that is induced by $A\to *$ and $\caO\to E_\infty$ induces an equivalence on
  group-completions.
\end{prop}

\begin{rmk}\label{rmk:p0con}
  Since $\caP(0)$ is contractible, the basepoint inclusion into each
  $\tilde\bP^p(*)$ is an equivalence, and hence the canonical map
  \[\caO(0)=\tilde\bO(*)=\B_0(\tilde\bO,\tilde\bE_0,*)\hookrightarrow
    |\B_\bullet(\tilde\bO,\tilde\bE_0,*)|\to
    |\B_\bullet(\tilde\bO,\tilde\bP,*)|=\caO\otimes^\bbL_{\caP} *\]%
  is an equivalence of $\caO$-algebras, so \cref{prop:B} in particular tells us
  \[\Omega\B(\caO\otimes^\bbL_\caP A)\simeq \Omega \B\caO(0)\times
    \Omega\B(E_\infty\otimes^\bbL_\caP A).\]
\end{rmk}

We point out that \cref{prop:B} is a variation of \cite[Thm.\,5.4]{BBPTY} and
\cite[Thm.\,5.9]{Bianchi-Kranhold-Reinhold}, and consequently, the proof is very
similar:\looseness-1

\begin{constr}
  We introduce a ‘stable’ version of $\tilde\bbO$: For each based space $X$, the
  space $\tilde \bbO(X)$ splits as $\smash{\coprod_g \tilde\bbO_g (X)}$, and we
  consider the mapping telescope\looseness-1
  \[\tilde \bbO_\infty(X)\coloneqq \on{hocolim}\left(\hspace*{-3px}
      \begin{tikzcd}\bbO_0(X)\ar{r}{\mathrm{stab}} &
      \tilde\bbO_1(X)\ar{r}{\mathrm{stab}} &
      \tilde\bbO_{2}(X)\ar[r] & \dotsb\end{tikzcd}\hspace*{-3px}\right).\]%
  Then the terminal map $X\to *$ of based spaces gives rise to maps
  $\tilde\bbO_g(X)\to \tilde\bbO_g(*)$ for each $g$, which are compatible with the
  stabilisations, and hence constitute a map among homotopy colimits
  $\tilde\bbO_\infty(X)\to \tilde\bbO_\infty(*)$. On the other hand, by fixing a
  path from the image of $\hat s$ inside $E_\infty(1)$ to $\mathbf{1}$, we obtain
  a homotopy from $\tilde\bbO_g(X)\to \tilde\bbE_\infty(X)$ to
  $\tilde\bbO_g(X)\to \tilde\bbO_{g+1}(X)$, followed by the map to
  $\tilde\bbE_\infty(X)$. We hence obtain a map out of the mapping telescope
  $\tilde\bbO_\infty(X)\to \tilde\bbE_\infty(X)$. The product map is called
  \[\Psi_X\colon \tilde\bO_\infty(X)\to \tilde\bO_\infty(*)\times
    \tilde\bE_\infty(X).\]
\end{constr}

The following statement is shown in the proof of \cite[Thm.\,5.4]{BBPTY} (it
does not use the second operad $\caP$ at all):

\begin{lem}\label{lem:key}
  For each based space $X$, the map $\Psi_X$ is a homology equivalence.
\end{lem}

We can now show \cref{prop:B} by adapting the proof of
\cite[Thm.\,5.9]{Bianchi-Kranhold-Reinhold}:

\begin{proof}[Proof of \cref{prop:B}]
  By replacing the algebra $A$ by $\caP\otimes^\bbL_\caP A$ and applying
  \cref{lem:freeWell} if necessary, we may assume that the algebra $A$ to start
  with is well-pointed, enabling us to employ the reduced monads to describe the
  pushforward.

  As $\caP(r)\to \caO(r)$ lands in $\caO_0(r)$ for each $r\ge 0$, each
  $\tilde\bbO_g$ is a itself right $\tilde\bbP$-functor and
  $\B_\bullet(\tilde\bO,\tilde\bP,A)$ decomposes into
  $\B_\bullet(\tilde\bbO_g,\tilde\bbP,A)$. Then the map $\Phi$ from the statement
  is a union of realisations of simplicial maps
  \[\Phi_g\colon \B_\bullet(\tilde\bO_g,\tilde\bP,A)\to
    \B_\bullet(\tilde\bO_g,\tilde\bP,*)\times\B_\bullet(\tilde\bE_\infty,\tilde\bP,A).\]
  Adding the propagator gives rise to maps $\tilde\bO_g\to\tilde\bO_{g+1}$ of
  $\tilde\bP$-functors, its homotopy colimit being the $\tilde\bP$-functor
  $\tilde\bO_\infty$. Since the maps $\Phi_g$ commute with these stabilisations,
  we obtain a colimit map $\Phi_\infty$. It suffices to show that $\Phi_\infty$
  is a homology equivalence, as this implies, by a classical telescope argument,
  that $\Phi$ induces an isomorphism among localised Pontrjagin rings
  $\smash{H_*\bigl(\caO\otimes_\caP^\bbL A\bigr)[\pi_0^{-1}] \to
    H_*\bigl((\caO\otimes^\bbL_\caP *)\times (E_\infty\otimes^\bbL_\caP
    A)\bigr)[\pi^{-1}_0]}$ and by the group-completion theorem, $\Omega \B\Phi$
  is a homology equivalence of loop spaces, and hence a weak equivalence as desired.

  To this end, we note that $\Phi_\infty$ is levelwise given by
  $\tilde\bO_\infty(\tilde\bP^pA)\to \tilde\bO_\infty(\tilde\bP^p*)\times
  \tilde\bE_\infty(\tilde\bP^pA)$.  Postcomposing with
  $\tilde\bO_\infty(\tilde\bP^p*\to *)\times\on{id}$, the resulting map agrees
  with $\Psi_{\smash{\tilde\bP^p A}}$, which is a homology equivalence by
  \cref{lem:key}.  However, since $\caP(0)$ is contractible, $\tilde\bP^p*\to *$
  is an equivalence of cofibrant spaces, and since $\tilde\bO_\infty$ preserves
  such equivalences, the map to postcompose with is an equivalence. This shows
  that $\Phi_\infty$ is levelwise a homology equivalence.  Finally, as both
  sides are proper, we can invoke the spectral sequence for ‘thick’ geometric
  realisations (as in the proof of \cref{thm:A}) to conclude that $\Phi_\infty$
  itself is a homology equivalence.
\end{proof}

\begin{cor}\label{cor:OHStang}
  Let $\theta\colon L\to\on{BO}(2n)$ be a spherical tangential structure
  with $n$-connected $L$, and let $A$ be an $\dE[\theta]{2n}$-algebra. Then we
  have a weak equivalence of loop spaces\looseness-1
  \[\Omega\B\dW[\theta]{*,1}[A] \simeq \Z\times\Omega^\infty_0\on{MT}\theta
    \times \Omega\B(\dE\infty \otimes^\bbL_{\dE[\theta]{2n}} A).\]
\end{cor}
\begin{proof}
  As explained in \cref{constr:Wgstab}, the left side silently stands for the
  group completion of the pushforward
  $\caW^\theta\otimes^\bbL_{\dF[\theta]{2n}}\kappa^* A$.  By \cref{prop:B}, this
  splits into the group-completion of $\caW^\theta(0)$, which is
  $\Z\times\Omega^\infty_0\MT\theta$ by \cref{cor:MT}, and the group completion
  of $E_\infty\otimes^\bbL_{\dF[\theta]{2n}}\kappa^* A$. However, since
  $\dE[\theta]{2n}\stackrel{\kappa}{\leftarrow}
  \dF[\theta]{2n}\stackrel{\pi_1}{\leftarrow} \dF[\theta]{2n}\times E_\infty\to
  E_\infty$ is a zig-zag of operads as in \cref{constr:Einf}, it follows that
  $E_\infty\otimes^\bL_{\dF[\theta]{2n}}\kappa^*A$ is a model for the homotopy
  type of the $E_\infty$-algebra $E_\infty\otimes^\bL_{\dE[\theta]{2n}} A$.
\end{proof}

\section{Calculations and examples}
The considerations of \cref{sec:split} reduce the original problem to
understanding, for each tangential structure $\theta\colon L\to \on{BO}(d)$ and
each $\dE[\theta]d$-algebra $A$, the spectrum
$\B^\infty(\dE\infty\otimes^\bbL_{\dE[\theta]d} A)$. We will give a general, yet
rather abstract answer to this question in \cref{sec:eqBar}, but we first
discuss several special cases, which rediscover results by
\cite{Bonatto-2020,Bonatto}.

\subsection{Suboperads and submonoids}

\begin{rmk}\label{rmk:Forward}
  Let $\caH$ be any proper operad mapping to $\dE[\theta]{2n}$. By
  \cref{rem:freeRes}, we obtain, for each $\caH$-algebra $X$, an equivalence
  $E_\infty\otimes^\bbL_{\dE[\theta]{2n}} (\dE[\theta]{2n}\otimes^\bbL_\caH
  X)\simeq E_\infty\otimes^\bbL_\caH X$ of $E_\infty$-algebras.  This shows that
  if $\theta\colon L\to \on{BO}(2n)$ is a spherical tangential structure with
  $n$-connected $L$, then we have a weak equivalence of loop spaces
  \[\Omega\B \dW[\theta]{*,1}[\dE[\theta]{2n}\otimes^\bbL_\caH X] \simeq
    \Z\times\Omega^\infty_0\MT\theta \times \Omega\B(E_\infty\otimes^\bbL_\caH X).\]
\end{rmk}

Particularly easy examples of $\caH$ are operads that arise from topological
monoids:

\begin{defi}
  Each topological monoid $M$ constitutes an operad $M_+$ with $M$ as its monoid
  of unary operations, with a single nullary operation, and with no operations
  of higher arity. Then $M_+$ is a proper operad if and only if $M$ is
  well-pointed. Moreover, $M_+$-algebras are the same as based $M$-spaces. For a
  based $M$-space $X$, we define its \emph{based homotopy quotient} as
  $X_\h{M}\coloneqq E_0\otimes^\bbL_{\smash{M_+}} X$. In the case where the
  $M$-action on $X$ is trivial, we have an equivalence
  $X_\h{M}\simeq (\B M)_+\wedge X$\looseness-1
\end{defi}

If $M$ maps to the monoid $\dE[\theta]{2n}$, then we obtain an operad map
$M_+\to \dE[\theta]{2n}$ by taking the unique nullary of $M_+$ to the unique
nullary of $\dE[\theta]{2n}$.

\begin{prop}\label{prop:Forward}
  Let $M$ be a well-pointed topological monoid mapping to
  $\dE[\theta]{2n}(1)$. Then we have, for each based $M$-space $X$, a weak
  equivalence of loop spaces
  \[\Omega\B \dW[\theta]{*,1}[\dE[\theta]{2n}\otimes^\bbL_{M_+} X]\simeq
    \Z\times\Omega^\infty_0\MT\theta\times\Omega^\infty\Sigma^\infty(X_\h{M}).\]%
\end{prop}
\begin{proof}
  By \cref{rmk:Forward}, it suffices to show that
  $\Omega\B(E_\infty\otimes^\bbL_{M_+} X)\simeq \Omega^\infty\Sigma^\infty
  (X_\h{M})$. To this end, we note that by \cref{constr:Einf}, each operad map
  $M_+\to E_\infty$ can be used to calculate the homotopy type of the left hand
  side. One operad map is given by the composition of $M_+\to E_0$, taking each
  unary operation in $M$ to the identity, and the inclusion $E_0\to
  E_\infty$. By \cref{rem:freeRes}, we get an equivalence of
  $E_\infty$-algebras
  \[E_\infty\otimes^\bbL_{M_+} X\simeq E_\infty\otimes^\bbL_{E_0}(E_0\otimes^\bbL_{M_+} X) \simeq E_\infty\otimes^\bbL_{E_0} X_\h{M}.\]
  We then conclude the proof by invoking the equivalence
  $\smash{\Omega \B (E_\infty\otimes^\bbL_{E_0}({-}))\simeq\Omega^\infty\Sigma^\infty}$
  on well-pointed spaces from \cite{Barratt-Priddy,Segal-1973} and noting that
  $X_\h{M}=E_0\otimes_{M_+}^\bbL X$ is well-pointed by \cref{lem:freeWell}.
\end{proof}

Our main example for a submonoid of $\dE[\theta]d$ comes from the following observation:

\begin{lem}\label{lem:BG}
  If $\theta$ is of the form $\B G\to \B\on{O}(d)$ for a well-pointed subgroup
  $G\subseteq \on{O}(d)$, then there is a zig-zag of equivalences of well-pointed
  monoids $G\leftarrow \tilde G\hookrightarrow \dE[\theta]{d}(1)$.
\end{lem}

We point out that there are models for $\dE[\theta]{d}$, for which $G$ is even a
strict subgroup of the monoid of unaries. For the model we use, this is not the
case.
  
\begin{proof}
  We start by noting that $\theta^*V_d$ is given by the Borel construction
  $\E G\times_G \R^{d}$. In this description, there is an element $e_0\in \E G$
  such that under the standard identification $D^{d}\times\R^{d}\cong TD^{d}$,
  we have $\ell_{0,1}(z,v)=[e_0,v]$ up to a fixed linear automorphism
  of $\R^{d}$.

  Now we let
  $\tilde G\subseteq \dE[\theta]{d}(1)=\Emb^\theta( D^d,
  D^d)$ be the submonoid consisting of all $(\alpha,\gamma,t)$ where
  $\alpha\colon  D^{d}\hookrightarrow  D^{d}$ is of the form
  $\omega|_{\smash{ D^d}}$ for some $\omega\in G\subset\on{O}(d)$ and where $\gamma$ %
  is of the form $s\mapsto ((z,v)\mapsto [\beta(s),v])$ for some path
  $\beta \colon [0,t]\to \E G$ from $e_0$ to $e_0\cdot \omega$. Then the map
  $f\colon\tilde G\to G$ taking $(\omega|_{\smash{ D^d}},\gamma,t)$ to
  $\omega$ is a homomorphism.
  
  In slightly different words, $\tilde G$ is a Moore model for the homotopy
  fibre of the map $G\to \on{E}G$ taking $\omega$ to $e_0\cdot \omega$ and $f$ is the
  canonical map from the homotopy fibre to $G$.  Since $\E G$ is contractible,
  this shows that $f$ is an equivalence. Moreover, we employ the equivalence
  $\Emb( D^d, D^d)\simeq \Bun_0(T
  D^d,T D^d)$ from the proof of \cref{lem:unaries} to see that the
  inclusion $\tilde G\hookrightarrow \dE[\theta]{d}(1)$ fits into the following
  morphism of Puppe sequences
  \[
    \begin{tikzcd}[column sep=1.5em]
      \tilde G\ar[r]\ar[d,hook] & G\ar[r]\ar{d}{\omega\mapsto T\omega|_{\smash{ D^d}}}[swap]{\omega\mapsto \omega|_{\smash{ D^d}}} & \E G\ar[r]\ar{d}{e\mapsto ((z,v)\mapsto [e,v])} & \B G\ar[d,equal]\\
      \dE[\theta]{d}(1)\ar[r] & \Emb( D^{d}, D^{d})\simeq \Bun_0(T D^d,T D^d)\ar[r]
      & \Bun(T D^{d},\E G\times_G\R^{d})\ar[r] & \B G,
    \end{tikzcd}
  \]
  so since the right vertical map is an equivalence, first one is an equivalence as well.
\end{proof}

\begin{expl}\label{ex:BG}
  Let $G\subset \on{O}(2n)$ be an $(n-1)$-connected well-based subgroup such
  that the tangential structure $\theta\colon\B G\to\B\on{O}(2n)$ is spherical.
  If $X$ is a based $G$-space, then it is in particular a based $\tilde G$-space
  and the based homotopy quotients $X_\h{G}$ and $X_\h{\tilde G}$ are
  equivalent. Therefore, \cref{prop:Forward} provides us with an
  equivalence\looseness-1
  \[\Omega\B\dW[\theta]{*,1}[\dE[\theta]{2n}\otimes^\bbL_{\tilde G_+} X]
    \simeq\Z\times\Omega^\infty_0\MT\theta\times\Omega^\infty\Sigma^\infty(X_\h{G}).\]%
  In the case of $2n=2$ and $\theta$ being $\B\SO(2)\to \B\on{O}(2)$, this
  recovers \cite[Cor.\,\textsc{d}]{Bonatto}.
\end{expl}

\subsection{Partial algebras}

A slightly more rigid viewpoint, which appears in both \cite{Bonatto} and
\cite{Kupers-Miller}, starts with a \emph{partially defined}
$\dE[\theta]d$-algebra $X$, and instead of taking the relatively free
$\dE[\theta]d$-algebra along an operad map to $\dE[\theta]d$, we
\emph{complete} the $\dE[\theta]d$-algebra:

\begin{defi}
  Let $\caP$ be an operad. A \emph{partial $\caP$-algebra} is a space $X$,
  together with $\Comp (X)\subseteq \bbP X$ and a map
  $\lambda\colon \Comp (X)\to X$ satisfying the axioms of \cite[Ex.\,41]{Kupers-Miller}.\looseness-1

  Clearly, a partial $\caP$-algebra $X$ with $\Comp(X)=\bbP X$ is the same as a
  $\caP$-algebra in the usual sense, hereafter called \emph{honest}
  $\caP$-algebra in order to avoid confusion.
\end{defi}

\begin{constr}
  Let $X$ be a partial $\caP$-algebra. We define an inverse system
  \[\left(\begin{tikzcd}\dotsb\ar[r,"c_2"] & \Comp_1(X)\ar[r,"c_1"] &\Comp_0(X)\end{tikzcd}\hspace*{-4px}\right),\]
  with $c_1=\lambda$ and $\Comp_p(X)\subseteq\bP\Comp_{p-1}(X)$ being the
  subspace of all $x$ satisfying $(\bbP c_{p-1})(x)\in \Comp_{p-1}(X)$; then
  $c_p$ is the restriction of $\bP c_{p-1}$ to that subspace.

  If $S$ is a right $\bbP$-functor with values in spaces, then we obtain a
  simplicial space $S\Comp_\bullet(X)$, where the degeneracy maps
  $s_i\colon S\Comp_p(X)\to S\Comp_{p+1}(X)$ are induced by the unit of $\bbP$,
  the face maps $d_i\colon S\Comp_p(X)\to S\Comp_{p-1}(X)$ for $0\le i\le p-1$
  are induced by the natural transformations $\bbP^2\to \bbP$ and $S\bbP\to S$,
  and the last face map $d_p$ is given by $Sc_p$, compare
  \cite[Lem.\,4.2]{Kupers-Miller}.  If $X$ is an honest $\caP$-algebra, then the
  simplicial space $S\Comp_\bullet(X)$ agrees with the usual two-sided bar
  construction $\B_\bullet(S,\bbP,X)$.\looseness-1
\end{constr}

\begin{constr}\label{constr:Completion}
  For a partial $\caP$-algebra $X$, we define its \emph{completion}
  $\hat X\coloneqq |\bbP\Comp_\bullet(X)|$, which is a $\caP$-algebra by
  $\bbP|\bbP\Comp_\bullet(X)|\cong |\bbP^2\Comp_\bullet(X)|\to
  |\bbP\Comp_\bullet(X)|$.  If $X$ is an honest $\caP$-algebra, then
  $\bbP\Comp_\bullet(X)$ is the $\bbP$-free simplicial resolution of $X$ from
  \cref{rem:freeRes}, so its realisation $\hat X=\caP\otimes^\bbL_\caP X$ is
  equivalent to $X$.

  If $S$ is any right $\bbP$-functor, then there is an equivalence among
  geometric realisations
  $|\B_\bullet(S,\bbP,\hat X)|\simeq |S\Comp_\bullet(X)|$, which is established
  as in the proof of \cite[Lem.\,44]{Kupers-Miller} by looking at the
  bisimplicial space $\B_\bullet(S,\bbP,\bbP\Comp_\bullet(X))$.
\end{constr}

\begin{constr}
  Let $W$ be a $\theta$-framed manifold and let $X$ be a partial
  $\dE[\theta]d$-algebra. Then we put
  $\int^\theta_W X\coloneqq \int^\theta_W\hat X$ and
  $W^\theta[X]\coloneqq W^\theta[\hat X]$, i.e.\ we first complete and then
  apply the old definitions. We have an equivalence\looseness-1
  \[\textstyle \int^\theta_W X= |\B_\bullet(\dbE[\theta]W,\dbE[\theta]d,\hat X)|
    \simeq |\dbE[\theta]W(\Comp_\bullet X)|,\] and similarly for
  $W^\theta[X]$, showing that in the case where $X$ is already an honest
  $\dE[\theta]d$-algebra, the two compe\-ting definitions are equivalent.
\end{constr}

Let $d=2n$. If $\caH\subseteq \dE[\theta]{2n}$ is a proper suboperad, and $X$ is an
$\caH$-algebra, regarded as a partial $\dE[\theta]{2n}$-algebra, then
$\hat X\simeq\dE[\theta]{2n}\otimes^\bbL_\caH X$, so if
$\theta\colon L\to \on{BO}(2n)$ is spherical with $n$-connected $L$, we are in
the situation of \cref{rmk:Forward}.
This leads us to the desired reformulation of \cref{prop:Forward} in terms of
partial algebras:\looseness-1

\begin{thm}\label{thm:C}
  Let $\theta\colon L\to \on{BO}(2n)$ be a spherical tangential structure with
  $n$-connected $L$. If $M\subseteq \dE[\theta]{2n}(1)$ is a well-pointed
  submonoid and $X$ is a based $M$-space, regarded as a partial
  $\dE[\theta]{2n}$-algebra, then we have a weak equivalence of loop spaces
  \[\textstyle\Omega\B \dW[\theta]{*,1}[X]
    \simeq \Z\times\Omega^\infty_0\MT\theta\times\Omega^\infty\Sigma^\infty(X_\h{M}).\]
\end{thm}

\subsection{Configuration spaces and punctured diffeomorphism groups}
\label{sec:lab}

Two further calculations are based on a reinterpretation of some cases
of factorisation homology $\int^\theta_W A$ as labelled configuration spaces
$C(\mathring W;X)$ studied in \cite{Segal-1973,CFB-1987}. This
reinterpretation appears to be well-known, but as I could not find a reference,I
spelled out a proof. Let us first repeat the definition from
\cite[§\,1]{CFB-1987} in our language:\looseness-1

\begin{defi}
  Let $Q$ be a space. For each $r\ge 0$, let $\tilde C_r(Q)\subset Q^r$ be the
  space of ordered configurations of $r$ particles in $Q$. Then we have a right
  $\bE_0$-functor $\bC_Q$ that takes a space $Y$ to
  $\coprod_r \tilde C_r(Q)\times_{\frS_r}Y^r$. For every based space $X$, we
  define the \emph{labelled configuration space} $C(Q;X)$ as the (strict)
  coequaliser of the two maps $\bC_Q\bE_0(X)\rightrightarrows \bC_Q(X)$.  If
  $X$ is well-pointed, this coequaliser is equivalent to the bar construction
  $|\B_\bullet(\bC_Q,\bE_0,X)|$.
\end{defi}

\begin{prop}\label{prop:lab}
  Let $X$ be a well-pointed space, considered as a based $\dE[\theta]d(1)$-space
  with trivial action, and hence as a partial $\dE[\theta]d$-algebra. Then we
  have weak equivalences $\textstyle\int^\theta_W X \simeq C(\mathring W;X)$ and
  $W^\theta[X]\simeq (C(\mathring W;X)\times
  \Fr^\theta_\partial(W,\ell_W))\sslash \Diff_\partial(W)$ that are compatible
  with embeddings.
\end{prop}
\begin{proof}
  We show the second equivalence, as the first one is proven analogously.
  Throughout the proof let $M$ be the monoid $\dE[\theta]d(1)$. We then have an
  equivalence
  \[W^\theta[X]=|\B_\bullet(\udbE[\theta]W,\dbE[\theta]d,\dE[\theta]d,\otimes^\bbL_{\tau
      \dE[\theta]d} X)|\sslash\Diff_\partial(W)
    \simeq|\B_\bullet(\udbE[\theta]W,\bbM_+,X)|\sslash\Diff_\partial(W).\]%
  We have fibrations
  $\ul\Emb^\theta(\ul{r}\times\mathring D^d,W)\to \tilde C_r(\mathring
  W)\times\Fr^\theta_\partial(W,\ell_W)$ with
  $(\ell,\alpha,\gamma)\mapsto(\alpha|_{\ul{r}\times \{0\}},\ell)$, inducing a
  transformation $\udbE[\theta]W\to \bC_W\times\Fr^\theta_\partial(W,\ell_W)$.
  We want to show that the $\Diff_\partial(W)$-equivariant augmentation
  $|\B_\bullet(\udbE[\theta]W,\bbM_+,\dbE0)|\to \bC_W\times
  \Fr^\theta_\partial(W,\ell_W)$ is an equivalence of right $\dbE0$-functors;
  then the claim follows from the chain of equivalence (all of which are
  compatible with postcomposing with $\theta$-framed embeddings)\looseness-1
  \begin{align*}
    \bigl(C(\mathring W;X)\times\Fr_\partial^\theta(W,\ell_W)\bigr)\sslash \Diff_\partial(W) &\simeq |\B_\bullet(\bC_W\times\Fr_\partial^\theta(W,\ell_W),\bE_0,X)|\sslash\Diff_\partial(W)
    \\&\simeq |\B_\bullet(\B_\bullet(\udbE[\theta]W,\bbM_+,\bE_0),\bE_0,X)|\sslash \Diff_\partial(W)
    \\&\simeq |\B_\bullet(\dbE[\theta]W,\bbM_+,X)|\sslash\Diff_\partial(W)
    \\&\simeq W^\theta[X].
  \end{align*}
  In order to reach this goal, we first note that for any trivial (unbased)
  $M$-space $Y$, the canonical map
  $\B_\bullet(\udbE[\theta]W,\bbM,Y)\to \B_\bullet(\udbE[\theta]W,\bbM_+,Y_+)$
  induces an equivalence on geometric realisations, so by 2-out-of-3, we can
  likewise show that for any such $Y$, the augmentation
  $f\colon \B_\bullet(\udbE[\theta]W,\bbM,Y)\to
  \bC_W(Y)\times\Fr^\theta_\partial(W,\ell_W)$ induces an equivalence on
  geometric realisations. Since all augmentation maps
  $\B_p(\udbE[\theta]W,\bbM,Y)\to \bC_W\times\Fr^\theta_\partial(W,\ell_W)$ are
  fibrations and all involved simplicial spaces are proper,
  \cite[Lem.\,2.14]{Ebert-RW} tells us that the homotopy fibre of $|f|$ is the
  geometric realisation of the actual simplicial fibre.\looseness-1

  To calculate this fibre, let
  $([w_1,\dotsc,w_r;y_1,\dotsc,y_r],\ell)\in (\tilde C_r(\mathring
  W)\times_{\frS_r} Y^r)\times\Fr^\theta_\partial(W,\ell_W)$ be any point. Since $Y$ is
  a trivial $M$-space, the actual fibre $F_\bullet$ is of the form
  \[F_p=\Emb^\theta_{w_1,\dotsc,w_r}(\ul{r}\times \mathring
    D^d,W)\times_{\frS_r}(M^p\times \{y_i\})^r_{i=1},\] where the
  $\theta$-framing $\ell$ of $W$ is taken into account and where
  $\Emb^\theta_{w_1,\dotsc,w_r}\subseteq \Emb^\theta$ is the subspace of those
  $\theta$-framed embeddings $(\alpha,\gamma)$ with $\smash{\alpha(i,0)=w_{\sigma(i)}}$
  for some $\sigma\in\frS_r$. We fix a $\theta$-framed embedding
  $(\bar\alpha,\bar\gamma)\colon \ul{r}\times \mathring D^d\hookrightarrow W$ with
  $\bar\alpha(i,0)=w_i$; then the map
  \begin{align*}
    \Emb^\theta_0(\mathring D^d,\mathring D^d)^r\times \frS_r&\to \Emb^\theta_{w_1,\dotsc,w_r}(\ul{r}\times\mathring D^d,W),\\
    (\alpha_1,\gamma_1,\dotsc,\alpha_r,\gamma_r;\sigma)&\mapsto (\bar\alpha,\bar\gamma)\circ ((\alpha_1,\gamma_r)\sqcup\dotsb\sqcup(\alpha_r,\gamma_r))\circ (\sigma\times\on{id}_{\smash{\mathring D^d}})
  \end{align*}
  is a deformation retract (the deformation given by shrinking the radii of the
  discs). The left hand side, in turn, is equivalent to $M^r\times\frS_r$
  itself.  Altogether, we have established a simplicial equivalence between the
  actual fibre $F_\bullet$ and the simplicial space $\B_\bullet(\bbM,\bbM,*)^r$,
  which is contractible as desired.
\end{proof}

\begin{expl}\label{ex:punct}
  Let $X$ be the based space $S^0$, considered as a trivial $\dE[\theta]{2n}(1)$-space,
  and hence as a partial $\dE[\theta]{2n}(1)$-algebra. By \cref{prop:lab}, we have an equivalence
  \[W^\theta[X]\simeq \coprod_{r\ge 0}(C_r(\mathring W)\times
    \Fr^\theta_\partial(W,\ell_W))\sslash\Diff_\partial(W),\]%
  where $C_r(\mathring W)$ is the \emph{unordered} configuration space of $r$
  particles inside $\mathring W$. Moreover, we have an action of
  $\Diff_\partial(W)$ on each $C_r(\mathring W)$ and the evaluation
  $\Diff_\partial(W)\to C_r(\mathring W)$ at a given configuration is a
  fibration, with the stabiliser being the subgroup $\Diff^r_\partial(W)$ of
  diffeomorphisms that fix a subset of cardinality $r$, allowing
  permutations. Therefore, the canonical map
  $*\sslash \Diff^r_\partial(W)\to C_r(\mathring W)\sslash
  \Diff_\partial(W)$ is a weak equivalence, so by a five lemma argument, we obtain
  an equivalence
  \[W^\theta[X]\simeq \coprod_{r\ge
      0}\Fr^\theta_\partial(W,\ell_W)\sslash\Diff^r_\partial(W).\] These are the
  punctured moduli spaces $\caM^{\theta,r}_\partial(W,\ell_W)$ considered in
  \cite{Bonatto-2020}. Using that $*_\h{\dE[\theta]{2n}(1)}\simeq (\B\dE[\theta]{2n}(1))_+\simeq L_+$, \cref{thm:C}
  provides us with an equivalence
  \[\Omega\B \coprod_{g,r\ge 0}\caM^{\theta,r}_\partial(W_{g,1},\ell_{g,1})\simeq \Z\times\Omega^\infty_0\MT\theta\times
    \Omega^\infty\Sigma^\infty_+L,\]%
  By noting that $\Omega^\infty\Sigma^\infty_+L$ is the group-completion of the
  labelled configuration space
  $C(\R^\infty;L_+)\simeq \coprod_{r\ge 0}L^r\sslash \frS_r$, we obtain an
  equivalence among Quillen plus-constructions
  \[\on{colim}_{g,r\to \infty} (\caM_\partial^{\theta,r}(W_{g,1},\ell_{g,1}))^+ \simeq \bigl(\on{colim}_{g\to \infty} \caM_\partial^\theta(W_{g,1},\ell_{g,1})\bigr)^+\times \bigl(\on{colim}_{r\to \infty}L^r\sslash \frS_r\bigr)^+.\]
  This is an instance of \cite[Thm.\,\textsc{a}]{Bonatto-2020}. We point out,
  however, that Bonatto’s result is stronger: In her case, one gets, very
  much as in \cite{CFB-Tillmann-2001}, a decoupling result for a fixed, finite
  number of punctures.
\end{expl}

\begin{expl}
  Recall the fibre sequence
  $\int_W^\theta A\to W^\theta[A]\to \caM^\theta_\partial(W,\ell_W)$ from
  \cref{prop:hofib}. In the case where $W=W_{g,1}$, we have stabilisation maps
  for all three terms in this sequence, and as they are compatible with each
  other, we reach a fibre sequence
  \[\textstyle \on{colim}_{g\to \infty} \int^\theta_{W_{g,1}}A\to
    \on{colim}_{g\to\infty}W^\theta_{g,1}[A]\to
    \on{colim}_{g\to\infty}\caM^\theta_\partial(W_{g,1},\ell_{g,1}).\]%
  In the case where $A$ is path-connected, the group-completion theorem and the main result
  of this paper provides a square
  \[
    \begin{tikzcd}
      \on{colim}_{g\to\infty}W^\theta_{g,1}[A]\ar[r]\ar[d] &
      \on{colim}_{g\to\infty}\caM^\theta_\partial(W_{g,1},\ell_{g,1})\ar[d]\\
      \Omega_0^\infty\MT\theta\times\Omega\B(E_\infty\otimes^\bbL_{\dE[\theta]{2n}}A)\ar[r,swap,"\pr"] &
      \Omega_0^\infty\MT\theta,
    \end{tikzcd}
  \]
  where the vertical maps are acyclic. One might be tempted to hope that the
  induced map on homotopy fibres
  \begin{align}\label{eq:fib}
    \textstyle\on{colim}_{g\to\infty}\int^\theta_{W_{g,1}}A\to
    \Omega\B(E_\infty\otimes^\bbL_{\dE[\theta]{2n}}A)
  \end{align}
  is an equivalence as well. Here is an example showing that this is not
  the case: Let $2n=2$ and consider the tangential structure
  $\theta\colon \B\on{SO}(2)\to \BO(2)$ of orientations. Moreover, consider the
  based space $X=S^3$, together with the trivial $\on{SO}(2)$-action. Under the
  identifications from \cref{prop:lab} and \cref{prop:Forward}, the map from
  \cref{eq:fib} is of the form
  \[\on{colim}_{g\to \infty} C(\mathring W_{g,1};S^3)\longrightarrow
    \Omega^\infty\Sigma^\infty(S^3\wedge\B\on{SO}(2)_+).\]%
  We want to show that this map does not induce an isomorphism in
  $H_5({-};\bF_2)$, and hence cannot be acyclic: On the one hand, we see that
  the right hand side is the free $E_\infty$-algebra over the based space
  $S^3\wedge \on{BSO}(2)_+\simeq \Sigma^3\CP^\infty_+$, so by
  \cite[Thm.\,\textsc{i}.4.1]{Cohen-Lada-May}, its rational homology is the free
  Dyer–Lashof algebra over based graded $\bF_2$-vector space
  $H_*(\Sigma^3\CP^\infty_+;\bF_2)$, which, in turn, is non-trivial in every odd
  degree. This shows that
  $H_5(\Omega^\infty\Sigma^\infty(S^3\wedge\B\on{SO}(2)_+);\bF_2)$ is non-trivial.
  Since homology commutes with filtered colimits, it suffices to show that
  $H_5(C(\mathring W_{g,1};S^3);\bF_2)$ is trivial for each $g\ge 0$. To this end,
  we note that the stable splitting from \cite[Prop.\,3]{CFB-1987} and the Thom
  isomorphism provides us with an isomorphism of graded $\bF_2$-vector spaces
  \[\tilde H_*(C(\mathring W_{g,1};S^3);\bF_2)\cong \bigoplus_{r\ge
      0}H_{*-3r}(C_r(\mathring W_{g,1});\bF_2).\]%
  Since $C_r(\mathring W_{g.1})$ is a non-compact $2r$-dimensional manifold, the
  shifted homology groups $H_{*-3r}(C_r(\mathring W_{g,1});\bF_2)$ are
  concentrated in degrees $3r\le *<5r$. Hence the claim follows from the fact
  that there is no integer $r\ge 0$ with $3r\le 5<5r$.
\end{expl}

\begin{rmk}
  In higher dimensions, the question whether the map from \cref{eq:fib} is
  acyclic or not is equivalent to the question whether the maximal perfect subgroup of
  $\on{colim}_{g\to \infty}\pi_1(\caM^\theta_\partial(W_{g,1},\ell_{g,1}))$ acts
  trivially on the fibre
  $\on{colim}_{g\to\infty}\int^{\smash\theta}_{\smash{W_{g.1}}}A$ or not.
  However, I am not aware of any work that has dealt with this question.
\end{rmk}

\label{sec:smaller}

\section{Relatively free \texorpdfstring{$\bm{E}_{\bm\infty}$}{E∞}-algebras}
In this section, we develop a general approach to the remaining question of the
homotopy type of $\B^\infty(\dE\infty\otimes^\bbL_{\dE[\theta]d} A)$ for a given
$\dE[\theta]d$-algebra $A$. By doing so, we assume familiarity with the language
of $\infty$-categories and $\infty$-operads as in \cite{Lurie-2009,Lurie}. Note
that the question we want to address is only about the operads $\dE[\theta]d$
and $E_\infty$ and has nothing to do with diffeomorphisms of manifolds any more.

\begin{constr}\label{rmk:bar}
  Recall from \cref{ex:unfr} the operad map $E_d\to \dE[\theta]d$. If $A$ is an
  $\dE[\theta]d$-algebra, then we denote by $U A$ its underlying $E_d$-algebra.
  For any $E_d$-algebra $\frA$, we consider the classical \emph{iterated bar
    construction} $\B^d \frA\coloneqq |\B_\bullet(\Sigma^d,\tilde \bbE_d,\frA)|$ as in
  \cite{May-1972}, which is a $(d-1)$-connected based space.\looseness-1
\end{constr}

\begin{rmk}
  Let $\caC$ be an $\infty$-category, let $L$ be a connected space and let
  $X\colon L\to \caC$ be a diagram.  Since $L$ is connected, the homotopy type
  (i.e.\ the equivalence class in $\caC$) of any $X(b_0)$ is independent of
  $b_0\in L$, and we say that $X$ \emph{takes value} $X(b_0)$.

  For each choice of basepoint $b_0\in L$, the diagram $X$ may be re-expressed
  in a slightly more classical language by the object $X(b_0)$ in $\caC$,
  together with an $E_1$-action by the loop space
  $\Omega L=\Omega_{\smash{b_0}}L$ on $X(b_0)$, and the colimit $\colim_L(X)$
  agrees with the (homotopy) quotient $X(b_0)_\h{\Omega L}$ if either side
  (hence both sides) exist(s) in $\caC$.
\end{rmk}

We point out that in the case where $L$ is part of a tangential structure
$\theta\colon L\to \on{BO}(d)$, then a canonical choice for a basepoint of $L$
would be the one coming from the bundle map
$\ell_{0,1}\colon T\mathring D^d\to \theta^*V_d$, whose map among base spaces factors
through a single point. The aim of this section is the following:

\begin{prop}\label{prop:C}
  Let $\theta\colon L\to\B\on{O}(d)$ be a tangential structure with connected
  $L$ and let $A$ be an $\dE[\theta]d$-algebra. Then the shifted suspension
  spectrum $\Sigma^{\infty-d}\B^d UA$ carries an $E_1$-action by the loop space
  $\Omega L$ and we have an equivalence of connective spectra
  \[\B^\infty(E_\infty\otimes^\bbL_{\dE[\theta]d} A)\simeq
  \bigl(\Sigma^{\infty-d}\B^d UA\bigr)_\h{\Omega L}.\]
\end{prop}

\begin{expl}
  If $L$ is contractible (i.e.\ $\dE[\theta]d\simeq E_d$), then
  $\B^\infty(E_\infty\otimes^\bbL_{\dE d} A)\simeq \Sigma^{\infty-d}\B^d A$.
\end{expl}

\begin{expl}
  If $\theta$ is of the form $\B G\to \B\on{O}(d)$ for some homomorphism
  $G\to\on{O}(d)$, then \cref{prop:C} gives rise to a naïve $G$-action on
  $\Sigma^{\infty-d}\B^d UA$, and we have an equivalence
  $\B^\infty(E_\infty\otimes^\bbL_{\dE[\theta]d} A)\simeq
  \bigl(\Sigma^{\infty-d}\B^d UA\bigr)_\h{G}$.
\end{expl}

Altogether, \cref{prop:C} then implies our main theorem:

\begin{thm}\label{thm:B}
  Let $\theta\colon L\to \B\on{O}(2n)$ be a spherical tangential structure with
  $n$-connected $L$, and let $A$ be an $\dE[\theta]{2n}$-algebra. Then there is
  an $A_\infty$-action of $\Omega L$ on the spectrum
  $\Sigma^{\infty-2n}\B^{2n} UA$ and we have a weak
  equivalence of loop spaces\looseness-3
  \[\textstyle \Omega\B\dW[\theta]{*,1}[A] \simeq
    \Z\times\Omega^\infty_0\MT \theta\times \Omega^\infty
    \bigl(\bigl(\Sigma^{\infty-2n}\B^{2n} U A\bigr)_\h{\Omega L}\bigr).\]
\end{thm}

In combination with \cref{thm:A}, we finally have achieved the
following homological computation:

\begin{cor}\label{cor:A}
  Let $\theta\colon L\to\on{BO}(2n)$ be a spherical tangential structure with
  $n$-connected $L$. If $2n\ge 6$ (or $2n=2$ and $\theta$ is admissible to
  \cite[Thm.\,7.1]{RW-2016}), we have, for each path-connected
  $\dE[\theta]{2n}$-algebra $A$ and each $g\ge 0$, isomorphisms
  \[H_i\bigl(\dW[\theta]{g,1}[A]\bigr) \cong
    H_i\bigl(\Omega_0^\infty\on{MT}\theta\times
    \Omega^\infty\bigl(\bigl(\Sigma^{\infty-2n}\B^{2n} U A\bigr)_\h{\Omega
      L}\bigr)\bigr)\]%
  for every $i$ small enough compared to $g$ to satisfy the conditions of \cref{thm:A}.
\end{cor}

In order to prove \cref{prop:C}, we have to reformulate parts of our operadic
framework in $\infty$-categorical terms:

\begin{nota}  
  For each $\infty$-category $\caC$ and objects $X,Y$, we denote the
  corresponding mapping space by $\Map_\caC(X,Y)$. The (large) $\infty$-category
  of $\infty$-categories is denoted by $\Cat_\infty$. For $\infty$-categories
  $\caC$ and $\caD$, we denote the $\infty$-category of functors $\caC\to\caD$
  by $\Fun(\caC,\caD)$.  We call the $\infty$-category of spaces $\caS$, the
  $\infty$-category of $d$-connective (i.e.\ $(d-1)$-connected) based spaces
  $\caS_*^{\ge d}$, and the $\infty$-category of connective spectra $\Sp^\conn$.
  
  The $\infty$-category of (coloured) $\infty$-operads is denoted by
  $\Op_\infty$.  Recall that $\infty$-operads are modelled in
  \cite[§\,2.1]{Lurie} as functors $\caO\to \Fin_*$ to the skeletal category of
  finite pointed sets satisfying certain conditions \cite[{}2.1.1.10]{Lurie}.  A
  map of $\infty$-operads $\caP\to\caO$ is a functor of categories over $\Fin_*$
  satisfying a further condition \cite[{}2.1.2.7]{Lurie}. Let
  $\Op_\infty(\caP,\caO)$ be the full subca\-tegory of
  $\Fun_{\Fin_*}(\caP,\caO)$ consisting of $\infty$-operad maps.
  We define the $\infty$-category $\Alg_\caO=\Alg_\caO(\caS)$ of $\caO$-algebras
  (in spaces) to be $\Op_\infty(\caO,\caS)$.\looseness-1
\end{nota}

\begin{rmk}
  Each topological operad is implicitly regarded as an $\infty$-operad through
  its operadic nerve \cite[{}2.1.1.23]{Lurie}.  For a $\frS$-free topological
  operad $\caO$, the $\infty$-category $\Alg_\caO$ of algebras over the operadic
  nerve of $\caO$ is equivalent to the coherent nerve of the simplicial model
  category of $\caO$-algebras in the sense of \cite{Berger-Moerdijk-2003}, see
  \cite[Thm.\,4.2.2]{Hinich-Moerdijk}, which is an instance of
  \cite[Thm.\,7.11]{Pavlov-Scholbach}.

  It then follows from \cite[{}5.2.4.6+7]{Lurie-2009} that if $\caP\to\caO$ is a
  map of $\frS$-free operads and $\caP$ is proper, then
  $\caO\otimes^\bbL_\caP({-})$ models the left-adjoint to the forgetful functor
  $\Alg_\caO\to \Alg_\caP$, which, in turn, is given by precomposing an operad
  map $\caO\to \caS$ with $\caP\to \caO$.
\end{rmk}

\begin{expl}
  The operad $\dE\infty$, called the ‘commutative $\infty$-operad’ in
  \cite[{}2.1.1.18]{Lurie}, is by definition the identity $\Fin_*\to\Fin_*$ and
  hence the terminal object in $\Op_\infty$. It agrees with the operadic nerve
  of the topological operad $E_\infty$, see \cite[{}5.1.1.5]{Lurie}.
\end{expl}

\begin{rmk}
  May’s recognition principle \cite{May-1972} can be formulated
  homotopy-coherently as in \cite[{}5.2.6.10]{Lurie}: Taking $d$-fold loop
  spaces becomes an equivalence of $\infty$-categories
  $\Omega^d\colon \caS^{\ge d}_*\to \Alg_{\dE d}$, and the $d$-fold bar
  construction $\B^d$ in the sense of \cref{rmk:bar} is a model for its
  essential inverse.
\end{rmk}

\begin{rmk}
  In \cite[§\,2.4.3]{Lurie}, the author constructs, out of an $\infty$-category
  $\caC$, an operad $\caC^\sqcup$, called the \emph{co-cartesian
    operad}. Intuitively, its colours are the objects of $\caC$, and its space
  of operations $(X_1,\dotsc,X_r)\to Y$ is given by $\prod_i
  \Map_\caC(X_i,Y)$. If $\caO$ is a \emph{unital} operad (i.e.\ for each object
  $X$, the space of nullary operations $()\to X$ is contractible), then we have a
  natural equivalence $\Op_\infty(\caO,\caC^\sqcup)\to\Fun(\caO\ula{1},\caC)$ of
  $\infty$-categories, where $(-)\ula{1}$ denotes the underlying
  $\infty$-category of unaries, see \cite[{}2.4.3.16]{Lurie}. This implies that
  $(-)^\sqcup$ is a right-adjoint functor from $\Cat_\infty$ to the full
  subcategory of $\Op_\infty$ spanned by unital $\infty$-operads, its
  left-adjoint given by $(-)\ula{1}$.

  As right-adjoints preserve terminal objects, the unital operad $E_\infty$ is
  of the form $*^{\sqcup}$, where $*$ is a point.  For two $\infty$-operads $\caO$
  and $\caO'$ and an $\infty$-category $\caC$, \cite[{}2.4.3.18]{Lurie}
  establishes an equivalence
  $\Op_\infty(\caC^\sqcup\times_{\Fin_*}\caO,\caO')\to
  \Fun(\caC,\Op_\infty(\caO,\caO'))$, which is natural in all arguments. For
  $\caO=\dE\infty$ and $\caO'=\caS$, we get an equivalence  of $\infty$-categories
  $\Alg_{\caC^\sqcup}\simeq \Fun(\caC,\Alg_{\dE\infty})$,
  which is natural in $\caC$.
\end{rmk}

\begin{nota}
  For any $\infty$-category $\caC$, precomposing functors $*\to \caC$ with the
  ter\-minal map $L\to *$ gives rise to the \emph{constant diagram functor}
  $\Delta_L\colon \caC\to \Fun(L,\caC)$. It is right-adjoint to
  $\colim_L\colon \Fun(L,\caC)\to \caC$ if the latter exists.\looseness-1
\end{nota}

\begin{lem}
  Let $Q$ be an $L^\sqcup$-algebra, regarded as a functor
  $L\to \Alg_{\dE\infty}$ as above. Then we have an equivalence of connective
  spectra
  \[\B^\infty(E_\infty\otimes^\bbL_{L^\sqcup} Q) \simeq
    \colim_L (\B^\infty Q).\]
\end{lem}
\begin{proof}
  By \cite[{}5.2.6.26]{Lurie}, the classical equivalence between the category
  $\Sp^\conn$ of connective spectra and group-like $E_\infty$-algebras has been
  made precise in the context of $\infty$-categories, and in
  \cite[§\,2.1]{Bunke-Tamme}, it has been shown that group-completion is
  left-adjoint to the inclusion of group-like $E_\infty$-algebras into the
  category of all $E_\infty$-algebras. This shows that
  $\B^\infty\colon \Alg_{\dE\infty}\to \Sp^\conn$ is a left-adjoint functor
  between $\infty$-categories and hence preserves colimits.  Second, under the
  equivalence $\Alg_{L^\sqcup}\simeq \Fun(L,\Alg_{E_\infty})$, the
  forgetful functor from $E_\infty$-algebras to $L^\sqcup$-algebras is
  identified with precomposing functors $*\to \Alg_{\dE\infty}$ with the
  terminal map $L\to *$, i.e.\ with the constant diagram functor $\Delta_L$. Its
  left-adjoint is given by taking colimits in the $\infty$-category of
  $E_\infty$-algebras.\looseness-1%
\end{proof}

\begin{proof}[Proof of \cref{prop:C}]
  Using \cite[Prop.\,2.2]{Horel-Krannich-Kupers}, there is a diagram
  $\Theta\colon L\to \Op_\infty$, which takes value $E_d$, whose colimit is
  equivalent to the operadic nerve of $\dE[\theta]d$, and for $b_0\in L$ coming
  from the $\theta$-framing $\ell_{0,1}$ of $D^d$, the inclusion
  $E_d\simeq \Theta(b_0)\to \on{colim}(\Theta)\simeq \dE[\theta]d$ is equivalent
  to the operad map from \cref{ex:unfr}.

  Moreover, we
  also we have an equivalence of $\infty$-operads
  $L^\sqcup\simeq \on{colim}_L(\Delta_LE_\infty)$, as one sees by applying the
  $\infty$-Yoneda lemma to the groupoid cores of the natural
  equivalence\looseness-1
  \begin{align*}
    \Op_\infty(L^\sqcup,{-}) &\simeq \Fun(L,\Op_\infty(E_\infty,{-}))
    \\&\simeq \on{lim}_L (\Op_\infty(\Delta_L E_\infty,{-}))
    \\&\simeq \Op_\infty(\on{colim}_L (\Delta_L E_\infty),{-}).
  \end{align*}
  As the $\infty$-category of $\infty$-operads has $L$-indexed colimits, the
  functor $\Delta_L$ is right-adjoint and hence preserves terminal objects. This
  shows that $\Delta_L E_\infty$ is a terminal diagram, so we have an
  (essentially unique) map $\Theta\to \Delta_LE_\infty$. Its colimit is an
  operad map $\dE[\theta]d\to L^\sqcup$, and since $E_\infty$ is terminal, the
  map $\dE[\theta]d\to E_\infty$ is homotopic to the composition
  $\dE[\theta]d\to L^\sqcup\to E_\infty$.
  This shows that we have a commutative diagram of right-adjoints
  \[
    \begin{tikzcd}[row sep=1.5em]
      \Fun(L,\Sp^{\on{cn}})\ar{d}[swap]{\Fun(L,\Omega^\infty)}\ar[rr,"\simeq"] && \lim_L(\Delta_L\Sp^{\on{cn}})\ar{d}{\lim_L (\Delta_L\Omega^\infty)}\\
      \Fun(L,\Alg_{E_\infty})\ar[r,"\simeq"] & \Alg_{L^\sqcup}\ar[r,"\simeq"]\ar{d}[swap]{(\dE[\theta]d\to L^\sqcup)^*} &\lim_L (\Alg_{\smash{\Delta_L E_\infty}})\ar{d}{\lim_L(\Theta\to \Delta_L E_\infty)^*}\\
      & \Alg_{\dE[\theta]d}\ar[r,swap,"\simeq"] &\lim_L(\Alg_\Theta),
    \end{tikzcd}
  \]%
  whence the corresponding diagram of left-adjoints commutes as well.  The
  basepoint $b_0\in L$ gives rise to projection functors
  $\pr{}_{b_0}\colon \kern-1px\lim_L (\Alg_\Theta)\to \Alg_{\smash{E_d}}$ etc.\
  out of the limit categories, resulting in a diagram (where
  $\smash{\Delta_L E_\infty\otimes^\bbL_{\smash\Theta}(-)}$ is the system of
  left-adjoints induced by the map $\Theta\to\Delta_L E_\infty$ of diagrams of
  $\infty$-operads)\looseness-1
  \[
    \begin{tikzcd}[row sep=1.5em]
      \Alg_{\dE[\theta]d}\ar[r,"\simeq"]\ar{d}[swap]{L^\sqcup\otimes^\bbL_{\dE[\scaleto{\theta}{3.6pt}]{\scaleto{d}{3.5pt}}}(-)} & \lim_L(\Alg_{\smash{\Theta}})\ar{r}{\pr_{b_{\scaleto{0}{3.5pt}}}}\ar{d}[swap]{\lim_L(\Delta_L E_\infty\otimes_{\smash \Theta}^\bbL(-))} & \Alg_{\smash{E_d}}\ar{d}{E_\infty\otimes^\bbL_{\smash{E_{\scaleto{d}{3.5pt}}}}(-)}\\
      \Fun(L,\Alg_{\smash{E_\infty}})\simeq \Alg_{L^\sqcup}\ar{d}[swap]{\Fun(L,\B^\infty)}\ar[r,"\simeq"] & \lim_L (\Alg_{\smash{\Delta_L E_\infty}})\ar{d}[swap]{\lim_L(\Delta_L\B^\infty)}\ar{r}{\pr{}_{b_{\scaleto{0}{3.5pt}}}} & \Alg_{\smash{E_\infty}}\ar[d,"\B^\infty"]\\
      \Fun(L,\Sp^{\on{cn}})\ar[r,"\simeq"] &
      \lim_L(\Delta_L\Sp^{\on{cn}})\ar{r}{\pr{}_{b_{\scaleto{0}{3.5pt}}}} &
      \Sp^{\on{cn}}.
    \end{tikzcd}
  \]
  Here the commutativity of the left squares is justified by the aforementioned
  diagram of right-adjoints, while the commutativity of the right squares is
  immediate as the central vertical maps are induced by morphisms of diagrams.
  We moreover point out the the top horizontal composition agrees up to
  equivalence with $\Op_\infty(\Theta(b_0)\to \on{colim}(\Theta),\caS)$ and
  hence with the ‘underlying $E_d$-algebra’ functor $U$ from \cref{rmk:bar}.

  Thus, for each $\dE[\theta]d$-algebra $A$, the left vertical
  composition $(\Fun(L,\B^\infty))(L^\sqcup\otimes^\bbL_{\dE[\theta]d} A)$,
  which is an $L$-indexed diagram in connective spectra, takes value
  $\B^\infty(E_\infty\otimes_{\dE d}^\bbL UA)$.
  We hence are left to determine, for a general $E_d$-algebra $\frA$, the spectrum
  $\B^\infty(E_\infty\otimes^\bbL_{\dE d} \frA)$. Here we see that for a given
  connective spectrum $Y$, we have a chain of natural equivalences of mapping
  spaces\looseness-1
  \begin{align*}
    \Map_{\Sp^\conn}(\B^\infty(E_\infty\otimes^\bbL_{E_d} \frA),Y)
    &\simeq \Map_{\Alg_{E_{\scaleto{d}{3.5pt}}}}(\frA,\Omega^d\Omega^{\infty-d}Y)
    \\ &\simeq \Map_{\smash{\caS^{\ge d}_*}}(\B^d \frA,\Omega^{\infty-d}Y)
    \\ &\simeq \Map_{\Sp^\conn}(\Sigma^{\infty-d}\B^d \frA,Y).
  \end{align*}
  This proves the statement by virtue of the $\infty$-Yoneda lemma.
\end{proof}

\begin{expl}\label{ex:loop}
  A \emph{$d$-connective retractive space} is a fibration $\xi\colon X\to L$
  with a section $\sigma$ such that the fibre $X_b$ over each $b\in L$
  is $(d-1)$-connected. For a compact $\theta$-framed manifold $(W,\ell)$ with
  boundary, we let $\Map^\theta_\partial(W,X)$ be the space of maps
  $s\colon W\to X$ with $\xi\circ s=\bar\ell$ and
  $s|_{\partial W}=\sigma\circ \bar\ell$, where $\bar\ell\colon W\to L$ is the
  map induced by $\ell$ among base spaces. Put differently, it is the space of
  sections of $\bar\ell^*X$ that coincide with $\bar\ell^*\sigma$ at
  $\partial W$.\looseness-1

  For the $\theta$-framed manifold $(D^d,\ell_{0,1})$, the space
  $\Omega^d_\theta X\coloneqq \Map^\theta_\partial(D^d,X)$ carries the structure
  of an $\dE[\theta]d$-algebra, generalising the usual $E_d$-structure on
  $\Omega^d X$ (here we use that $\bar\ell_{0,1}$ is constant with value $b_0$). Its underlying
  $E_d$-algebra $U\Omega^d_\theta X$ is equivalent to $\Omega^d X_{\smash{b_0}}$.
  Generalising the scanning map from \cite[Prop.\,2]{CFB-1987},
  \emph{non-abelian Poincaré duality} \cite[Thm.\,5.5.6.6]{Lurie}
  establishes an equivalence between $\int^\theta_W\Omega^d_\theta X$ and
  $\Map^\theta_\partial(W,X)$.

  As before, we would like to consider the \emph{moduli space} of such objects,
  and we do so by letting the $\theta$-framing vary: Let $\Map_\partial(W,L)$ be
  the space of maps $f\colon W\to L$ with
  $f|_{\partial W}=\bar\ell|_{\partial W}$, and let $\Map_\partial(W,X)$ be the
  space of maps $g\colon W\to L$ with
  $g|_{\partial W}=\sigma\circ \bar\ell|_{\partial W}$. Then we consider the homotopy pullback
  \[
    \begin{tikzcd}
      \ul{\smash\Map}^\theta_\partial(W,X)\ar[r]\ar[d]\arrow[dr, phantom, "\lrcorner", very near start] & \Map_\partial(W,X)\ar[d]\\
      \Fr^\theta_\partial(W,\ell)\ar[r] & \Map_\partial(W,L)
    \end{tikzcd}
  \]
  The space $\ul{\smash{\Map}}^\theta_\partial(W,X)$ carries a
  $\Diff_\partial(W)$-action, and the above equivalence can be enhanced to a
  $\Diff_\partial(W)$-equivariant equivalence
  $|\B_\bullet(\udbE[\theta]W,\dbE[\theta]d,\Omega^d_\theta X)|\simeq
  \ul{\smash{\Map}}^\theta_\partial(W,X)$. We
  hence obtain an equivalence
  \[W^\theta[\Omega^d_\theta X]\simeq
    \ul{\smash{\Map}}^\theta_\partial(W,X)\sslash \Diff_\partial(W).\] The right
  side can be interpreted as the moduli space of $\theta$-framed manifolds of
  type $(W,\ell_W)$, together with a map to the retractive space $X$,
  abbreviated by $\caM^\theta_\partial(W,\ell_W)\ula{X}$.
\end{expl}

Coming back to our situation of $d=2n$ and $(W,\ell_W)=(W_{g,1},\ell_{g,1})$, we
use the equivalence
$B^{2n}U\Omega^{2n}_\theta X=B^{2n}\Omega^{2n}X_{\smash{b_0}}\simeq
X_{\smash{b_0}}$ to conclude via \cref{thm:B} the following:

\begin{cor}\label{cor:e}
  Let $\theta\colon L\to\on{BO}(2n)$ be a spherical tangential structure with
  $n$-connected $L$. Moreover, let $X$ be a $2n$-connective retractive space. Then we have
  a weak equivalence
  \[\on{colim}_{g\to \infty}\bigl(\caM^\theta_\partial(W_{g,1},\ell_{g,1})\ula{X}\bigr)^+\simeq\Omega^\infty_0\MT\theta\times\Omega^\infty\bigl((\Sigma^{\infty-2n}X_{b_0})_\h{\Omega L}\bigr).\]
\end{cor}

\begin{outl}
  It would be interesting to have, for an $\dE[\theta]d$-algebra $A$, a
  description for the action of $\Omega L$ on $\Sigma^{\infty-d}\B^d UA$.  I
  expect the following to hold: For each euclidean $d$-dimensional vector space
  $V$, there is an operad $E_V$ equivalent to $E_d$, defined exactly as in
  \cref{ex:Ed}, and a $d$-fold bar construction
  $\B^V\colon \Alg_{E_V}\to \caS^{\ge d}_*$.  If $\theta\colon L\to \on{BO}(d)$
  is a tangential structure, then we obtain the diagram
  $\Theta\colon L\to\on{Op}_\infty$ from
  \cite[Prop.\,2.2]{Horel-Krannich-Kupers} by applying the construction
  $V\mapsto E_V$ fibrewise to the euclidean vector bundle
  $\theta^*V_d$. By applying $V\mapsto \B^V$
  fibrewise, we then should obtain a morphism $\Alg_\Theta\to \Delta_L\caS_*^{\ge d}$ of
  $L$-indexed diagrams in $\Cat_\infty$. Passing to the limit, this gives rise to a 
  functor
  \[\textstyle\B^d_\theta\colon \Alg_{\dE[\theta]d}\to \on{Fun}(L,\caS^{\ge d}_*),\]
  which we shall call the \emph{$\theta$-framed bar construction}. If $\theta$
  comes from a subgroup $G\subseteq\on{O}(d)$, then I expect this construction
  to coincide with the ($1$-categorical) $G$-equivariant bar construction from
  \cite{Salvatore-Wahl}. For each $d$-connective retractive space $X$ over $L$,
  I expect $\B^d_\theta\Omega^d_\theta X$ to be the functor corresponding to the
  fibration $X\to L$.

  The fibrewise one-point compactification $\theta^*V_d$ gives rise to a diagram
  $\theta^* V_d^\infty\colon L\to\mathcal{S}^{\ge d}_*$, and I
  expect that a variation of the above Yoneda argument shows that
  $\B^\infty(L^\sqcup\otimes^\bbL_{\dE[\theta]d} A)$ is equivalent to the
  $L$-parametrised mapping spectrum
  $\on{map}(\Sigma^\infty\theta^*V_d^\infty,\Sigma^\infty\B^d_\theta A)$.

  It seems to me that a formal argument for such a description would require a
  deeper analysis of the $\infty$-categorical bar construction and its
  naturality, and hence would go beyond the scope of this article.
\end{outl}

\label{sec:eqBar}

\printbibliography[heading=bibintoc]

\addr{Florian Kranhold}{%
  Karlsruhe Institute of Technology,
  Englerstraße 2,
  76131 Karlsruhe,
  Germany,
  \mail{kranhold@kit.edu}.}

\end{document}